\documentclass[12pt,a4paper]{amsart}
\usepackage{setspace}
\usepackage{hyperref}
\usepackage{amsfonts}
\usepackage[left=38mm,top=32mm,right=28mm,bottom=33mm]{geometry}
\usepackage{amssymb}
\usepackage{latexsym}
\usepackage{amsmath}
\usepackage{mathrsfs}
\usepackage{epigraph}
\usepackage[all,2cell]{xy}
 \UseAllTwocells \SilentMatrices
\usepackage{ifthen}
\usepackage{comment}

\usepackage{array,graphics,color}

\newenvironment{gap}
  {\color{blue}}%
  {}%

\newenvironment{nap} 
  {\color{red}}%
  {}%

\newtheorem{defn}{Definition}[section]
\newtheorem{lem}[defn]{Lemma}
\newtheorem{thm}[defn]{Theorem}
\newtheorem{cor}[defn]{Corollary}
\newtheorem{ex}[defn]{Example}

\newtheorem{prop}[defn]{Proposition}
\newtheorem{rem}[defn]{Remark}
\newtheorem{rems}[defn]{Remarks}

\newcommand{\noisetwo}{\noise^{\oplus 2}}
\newcommand{\Kiltwo}{\Kil^{\oplus 2}}

\newcommand{\tA}{\tilde{\mathcal{A}}}

\newcommand{\noise}{\mathsf{k}}

\newcommand{\open}{\mathcal{O}}
\newcommand{\A}{\mathcal{A}}

\newcommand{\Kil}{\mathsf{K}}
\newcommand{\kil}{\mathsf{k}}

\newcommand{\Real}{\mathbb{R}}
\newcommand{\Rplus}{\Real_+}

\newcommand{\Comp}{\mathbb{C}}
\newcommand{\Nat}{\mathbb{N}}

\newcommand{\B}{\mathcal{B}}

\newcommand{\tr}{\mathrm{Tr}}

\newcommand{\wt}{\widetilde}
\newcommand{\ol}{\overline}

\newcommand{\ot}{\otimes}

\newcommand{\opower}[1]{\ensuremath{^{\otimes #1}}}

\newcommand{\op}{\oplus}

\newcommand{\m}{\ensuremath{\mathrm{M}}}

\newcommand{\tu}{\textup}

\DeclareMathOperator{\Dom}{Dom}

\DeclareMathOperator{\Lin}{Lin}
\DeclareMathOperator{\Ker}{Ker}

\DeclareMathOperator{\re}{Re}
\DeclareMathOperator{\im}{Im}
\DeclareMathOperator{\diag}{diag}

\newcommand{\ip}[2]{\langle #1, #2 \rangle}

\newcommand{\norm}[1]{\lVert #1 \rVert}

\DeclareMathOperator*{\slim}{s-lim}

\newcommand{\en}{E$_0$}
\newcommand{\munit}{$\mu$nit~}
\newcommand{\munits}{$\mu$nits~}
\newcommand{\twoone}{II$_1$~}

\newcommand{\munitset}[1]{\ensuremath{\mathcal{U}_{#1,#1'}}}

\newcommand{\Ind}{\ensuremath{\mathrm{Ind}}}
\newcommand{\factor}{\mathrm{M}}
\renewcommand{\H}{\ensuremath{\mathcal{H}}}

\newcommand{\n}{\ensuremath{\mathrm{N}}}

\newcommand{\F}{\ensuremath{\mathfrak{F}}}

\newcommand{\semiflowalg}{\ensuremath{\mathcal{A}}}

\newcommand{\E}{\ensuremath{\mathcal{E}}}
\newcommand{\Iset}{\ensuremath{\mathcal{I}}}
\newcommand{\Jset}{\ensuremath{\mathcal{J}}}

\def\r{\mathcal{R}}

\newenvironment{rlist}
{

\begin{enumerate}}
{\end{enumerate}}

\title[Non-cocycle-conjugate \en-semigroups on factors]{Non-cocycle-conjugate \en-semigroups on factors}
\author[O. Margetts]{Oliver T. Margetts}
\address{Department of Mathematics and Statistics,
 Fylde College,
Lancaster University, Lancaster LA1 4YF, U.K.}
\email{o.margetts@lancaster.ac.uk}

\author[R. Srinivasan]{R. Srinivasan}
\address{Chennai Mathematical Institute, H1, SIPCOT IT Park, Kelambakkam, Siruseri 603103, India.}
\email{vasanth@cmi.ac.in}

\subjclass[2010]{Primary  46L53; Secondary 46L40, 46L55}
 \keywords{*-endomorphism, E$_0$-semigroup, CAR algebra, CCR algebra, quasifree state, super product system, type III factor, type II$_\infty$ factor}

\begin{document}

\begin{abstract}
We investigate \en-semigroups on general factors, which are not  necessarily of type I, and analyse associated invariants like product systems, super product systems etc. By tensoring \en-semigroups on type I factors with \en-semigroups on type II$_1$ factor, we produce several families (both countable and uncountable), consisting of mutually non-cocycle-conjugate of \en-semigroups on the hyperfinite II$_\infty$ factor. Using CCR representations associated with quasi-free states,  we construct for the first time, uncountable families consisting of mutually non-cocycle-conjugate \en-semigroups on all type III$_\lambda$ factors, for $\lambda \in (0,1]$. 
\end{abstract}

\maketitle

\section{Introduction}\label{intro}
\en-semigroups are semigroups of normal unital $*-$endomorphisms on a von Neumann algebra, which are $\sigma$-weakly continuous. They arise naturally in the study of open quantum systems, the theory of interactions,  algebraic quantum field theory, and in quantum stochastic calculus. The study of \en-semigroups lead to the study of interesting objects like product systems, super product systems, $C^*-$semiflows as its associated invariants.

For \en-semigroups on type I factors the subject has grown rapidly since its inception in \cite{Pow Pow}. We refer to the monograph \cite{Arveson book} for an extensive treatment regarding the theory of \en-semigroups on type I factors. Arveson showed that \en-semigroups on type I factors are completely classified by continuous tensor products of Hilbert spaces, called \emph{product systems}. This gives a rough division of \en-semigroups into three types, namely I, II and III. The type I \en-semigroups on type I factors are cocycle conjugate to the CCR flows (\cite{Arveson book}), but there are uncountably many exotic \en-semigroups of types II and III ( \cite{pdct}, \cite{T1}, \cite{genccr}, \cite{toepcar} \cite{lieb}) on type I factors. 

There has been relatively little progress regarding the study of \en-semigroups on type II$_1$ factors, after it was initiated in the 1988 paper \cite{Pow Pow}.  In \cite{alevras} Alexis Alevras introduced an index using Powers' boundary representation (\cite{Pow Pow}), and computed the index for several important cases. Still, this did not classify even the simplest examples of \en-semigroups on the hyperfinite II$_1$ factor called Clifford flows, since it is yet not proved that the Powers-Alevras index is a cocycle conjugacy invariant. The problem of non-cocycle-conjugacy for Clifford flows is resolved in \cite{MS}, even though it is still open to prove that the boundary representation is invariant under cocycle conjugacy.
 
 In \cite{MS} four new cocycle conjugacy invariants for $E_0$-semigroups on II$_1$ factors, namely a coupling index, a dimension for the gauge group, a \emph{super product system} and a $C^*$-semiflow were introduced, and computed for standard examples.  Using the $C^*$-semiflow and the boundary representation of Powers and Alevras, it was shown that the families of Clifford flows and even Clifford flows contain mutually non-cocycle-conjugate \en-semigroups. 
 
 On the other hand there is nearly no work done regarding \en-semigroups on type II$_\infty$ factors and type III factors. There is lot of work done in the frame work of product system of Hilbert modules introduced by Michael Skeide, with contributions from people like B.V.R. Bhat. But this theory of product system of Hilbert modules doesn't seem to be helpful in distinguishing the concrete examples of \en-semigroups we deal in this paper. In this paper for the first time we produce uncountable families containing mutually non-cocycle-conjugate \en-semigroups on the hyperfinite type II$_\infty$ factor and on all type III$_\lambda$ factors for $\lambda \in (0,1]$.
 
 This paper is structured as follows. In Section 2 we fix our notations and  give the basic definitions of  \en-semigroup, and notions of  cocycle conjugacy, units and the gauge group. We recall the definitions of important families of \en-semigroups namely CCR flows, generalized CCR flows, Toeplitz CAR flows on type I factors, and  Clifford flows, even Clifford flows on the hyperfinite II$_1$ factor. We also recall some important results regarding these families.
 
In Section 3, we generalize the definition of coupling index to \en-semigroups  on general factors, which was initially defined for \en-semigroups on type II$_1$ factors. After proving it is well-defined,  we clarify its relationship to the Powers-Arveson index for \en-semigroups on type I factor. 

In Section 4, we generalize the main result of \cite{alevras} for \en-semigroups on II$_1$ factors to \en-semigroups on any general factor, the association of product systems of Hilbert modules as a complete invariant. We use the frame of von Neumann modules, introduced in \cite{BMSS}, which is proved to be equivalent to the framework of \cite{skeide}. 

In Section 5, we associate a super product system to \en-semigroups  on general factors, which was initially defined for \en-semigroups on type II$_1$ factors, and show that this association is invariant under cocycle conjugacy. We also prove that the super product system of tensor product of \en-semigroups is the tensor product of the super product systems of the corresponding \en-semigroups, a fact which has been already proved for the case when the factor is type I or type II$_1$. 

In Section 6, we produce \en-semigroups on type II$_\infty$ factors by tensoring \en-semigroups on type I factor with \en-semigroups on type II$_1$ factors, and study the problem of non-cocycle-conjugacy. We prove that  a tensor product of a CCR flow of index $m$ with a Clifford flow (or with an even Clifford flow) of index $n$ is cocycle conjugate to another tensor product of a CCR flow of index $p$ with a Clifford flow (or with an even Clifford flow) of index $q$ if and only if $(m,n)=(p,q)$. Then we produce uncountable families of non-cocycle-conjugate \en-semigroups on the type II$_\infty$ factor by fixing either a Clifford flow or an even Clifford flow on the hyperfinite II$_1$ factor and tensoring with many  families  containing mutually non-cocycle-conjugate  type III \en-semigroups on type I$_\infty$ factor.

In Section 7, we analyze the \en-semigroups on type III factors, constructed using CCR representations associated with a quasi-free state corresponding to a complex linear positive operator $A\geq 1$, such that $A-1$ is injective. Since it is given by a Toeplitz operator, we call them as Toeplitz CCR flows on type III factors. We show that these Toeplitz CCR flows are equi-modular with respect to the invariant vacuum state (as defined in \cite{BISS}) if and only if the quasi-free state is given by an  operator of the form $A=1\ot R$ on $L^2(0,\infty) \ot \kil$. In this simplest case, we refer to these Toeplitz CCR flows as just CCR flows on type III factors given by $R$. We prove that these CCR flows are canonically extendable (which was defined as extendable in \cite{BISS}), and they canonically extend to CCR flows (on type I factors) of index equal to the rank of $R$. From this it follows that CCR flows associated with operators of the form $A=1\ot R$, with $R$ having different ranks, are not cocycle conjugate.

In Section 8, we further analyze the CCR flows given by positive operators of fixed rank. We prove that two such CCR flows are cocycle-conjugate if and only if they are unitarily equivalent. This in consequence produce uncountably many mutually non-cocycle-conjugate \en-semigroups on all type III$_\lambda$ factors for $\lambda \in (0,1]$. 

 \en-semigroups can also be constructed on type III factors using the CAR representations. 
But it can be proven that they are not canonically extendable (see \cite{Bk}). Since canonical extendability is a property invariant under cocycle conjugacy it follows that  none of the CAR flows are cocycle conjugate to the canonically extendable CCR flows on type III factors. 

At present the definition of `types' for \en-semigroups on general factors is not very clear. For instance the following reasons contribute to the confusions. If we define type I condition as the property of the product system of Hilbert bimodules being generated by its units (in the strong topology) as a bimodule, then every \en-semigroup on a type II$_1$ factor will satisfy that condition. A different definition of `types' for \en-semigroups on II$_1$ factor is given in \cite{MS}. But that definition is also not a satisfactory one, since every \en-semigroup on a II$_1$ factor will satisfy that definition. This follows from Lemma 8.3, in \cite{MS}, after some minor computations. On the other hand, the traditional notion of spatiality for \en-semigroups on type I factor (as defined by Arveson-Powers) means it is multi-spatial. The type III (non-spatial) examples like generalized CCR flows and Toeplitz CAR flows indeed admit a unit in their product system of Hilbert modules.  So we can not define spatiality for \en-semigroups on general factors as the property of just admitting a unit in the product system of Hilbert modules, if it has to be consistent with the existing definitions for \en-semigroups on type I factors. 

But whatever be the definition of type I for \en-semigroups on general factors, the CCR flows on type III factors (given by quasi-free states associated with operators of the form $I_{L^2(\Rplus)} \ot R$) are among the simplest kind of \en-semigroups, with associated super product systems as type I (Arveson) product systems.

\section{Preliminaries}\label{prelims}
 In this section we cover some of the necessary theory needed to study \en-semigroups on factors, and also fix our notations.
  
{\bf Notations:} $\Nat$ denotes the set of natural numbers, and we set $\Nat_0=\Nat\cup \{0\},$ $\overline{\Nat}=\Nat\cup \{\infty\}$. For any real Hilbert space $G$, we denote the complexification of $G$  by $G^\Comp$. Throughout this paper, we use the symbol $\kil$ to denote a separable real Hilbert space with $dim(\kil)\in \overline{\Nat}$, except in Sections \ref{CCR section} and \ref{CCR section2}. In Sections \ref{CCR section} and \ref{CCR section2},  $\kil$ is a complex Hilbert space, which is mentioned there. For any measurable subset $S\subseteq \Real$, $L^2(S, \kil)$ is the Hilbert space of square integrable functions on  $S$ taking values in $\kil$.

$H$ will always represent a complex Hilbert space and by $\ol{H}$ we denote the dual space anti-isomorphic to $H$. The inner product is always conjugate linear in the first variable and linear in the second variable.  For $E \subseteq B(H)$, we shall write $\left[E\right]$ for the closure of the linear subspace of $B(H)$ spanned by
$E$, in the weak operator topology. Similarly, if $S \subset H$ is a subset of vectors, we shall
write $\left[S\right]$ for the norm-closed subspace of $H$ spanned
by $S$.   By $\Lin(S)$we denote the linear subspace spanned by the set $S$, without taking the closure. 
For subspaces $E,F \subseteq B(H)$, $$EF=\{xy: x\in E, y\in F\};~E^*=\{x^*: x\in E\}.$$ For $E\subseteq B(H)$, $F\subseteq B(K)$, $E\ot F=\left[\{x\ot y: x\in E, y\in F\}\right].$ 

A subset $\open\subseteq [0,a]$ is an elementary set, if $\open=\cup_{n=1}^N(s_n, t_n)$ a finite disjoint union of open intervals. We assume $s_{n+1}>t_n$. By $\open^c$ we mean the interior of the complement in $[0,a]$. For a Borel set $E\subseteq \Real$, $|E|$ denotes the Lebesgue measure of $E$.

For von Neumann algebras $\m$ and $\n$, we denote by $\m\vee \n$ the von Neumann algebra generated by $\m$ and $\n$. For von Neumann algebras $\m_1, \m_2, \n_1, \n_2$ the following relation  holds; $$\left( \m_1 \ot \n_1\right)\vee \left(\m_2\ot \n_2\right) =\left(\m_1\vee \m_2\right)\ot \left(\n_1\vee \n_2\right);$$ By taking commutants we get the dual version $$\left( \m_1 \ot \n_1\right)\cap \left(\m_2\ot \n_2\right) =\left(\m_1\cap \m_2\right)\ot \left(\n_1\cap \n_2\right).$$  We call this as the distributive property of the tensors.

\bigskip

\begin{defn}
An \en-semigroup on a von Neumann algebra $\m$ is a semigroup $\{\alpha_t: t\geq0\}$ of normal, unital *-endomorphisms of $\m$ satisfying
\begin{itemize}
 \item[(i)] $\alpha_0=id$,
 \item[(ii)] $\alpha_t(\factor)\neq\factor$ for all $t>0$,
 \item[(iii)] $t\mapsto \rho(\alpha_t(x))$ is continuous for all $x\in\factor$, $\rho\in\factor_*$.
\end{itemize}
\end{defn}

\begin{defn} A cocycle for an \en-semigroup $\alpha$ on $\m$ is a strongly continuous family of unitaries $\{U_t: t\geq 0\}\subseteq \m$ satisfying $U_s\alpha_s(U_t)=U_{s+t}$ for all $s,t\geq0$.
\end{defn}

For a cocycle $\{U_t: t\geq 0\}$, we automatically have $U_0=1$. Furthermore the family of endomorphisms $\alpha_t^U(x):=U_t\alpha_t(x)U_t^*$ defines an \en-semigroup. This leads to the following equivalence relations on \en-semigroups.

\begin{defn}\label{conjugacy def} Let $\alpha$ and $\beta$ be \en-semigroups on von Neumann algebras $\m$ and $\n$. 
 \begin{itemize}
  \item[(i)] $\alpha$ and $\beta$ are \emph{conjugate} if there exists a *-isomorphism $\theta:\m\to\n$ such that $\beta_t=\theta\circ\alpha_t\circ\theta^{-1}$ for all $t\geq0$.
  \item[(ii)] $\alpha$ and $\beta$ are \emph{cocycle conjugate} if there exists a cocycle $\{U_t: t\geq 0\}$ for $\alpha$ such that $\beta$ is conjugate to $\alpha^U$. 
 \end{itemize}
\end{defn}

 
 Two \en-semigroups $\alpha$ and $\beta$, acting on $\m\subseteq B(H_1)$ and $ \n\subseteq B(H_2)$ respectively, are said to be spatially conjugate if there exists a unitary $U:H_1\mapsto H_2$ satisfying
\begin{itemize}
 \item[(i)] $U\m U^*=\n$,
 \item[(ii)] $\beta_t(x)=U\alpha_t(U^*xU)U^*$ for all $t\geq0$, $x\in\n$,
\end{itemize}
 
We say a von Neumann algebra $\m$ is in standard form if $\m\subseteq B(H)$ has a cyclic and separating vector $\Omega \in H$, called as the vacuum vector. 
Without loss of generality we can assume that an \en-semigroup is acting on a von Neumann algebra in a standard form, thanks to the following lemma.

 \begin{lem}\label{spatial conjugacy lemma}
 Let $\alpha$ be an \en-semigroup on the von Neumann algebra $\m$. Then $\alpha$ is conjugate to an \en-semigroup $\beta$ on a von Neumann algebra $\n$ in standard form. Moreover $\beta$ is unique up to spatial conjugacy. 
 \end{lem}
  
 \begin{proof}
Pick a faithful normal state $\varphi$ and let $(\pi_\varphi,H_\varphi,\Omega_\varphi)$ be the corresponding GNS triple. Then the \en-semigroup $\beta=\pi_\varphi\circ\alpha\circ\pi_\varphi^{-1}$ on $\pi_\varphi(\m)$ will suffice. Let $\n\cong\m$ be a von Neumann algebra acting standardly on $H_\psi$ with cyclic and separating vector $\Omega_\psi$ and corresponding state $\psi$. If $\gamma$ is an \en-semigroup on $\n$, conjugate to $\alpha$ via $\Phi:\m\to\n$, then we obtain a faithful normal state $\psi\circ\Phi\circ\pi_\varphi^{-1}$ on $\pi_\varphi(\m)$. By \cite{Araki modular} there exists a cyclic and separating vector $\Omega_\Phi$ in $H_\varphi$ implementing $\psi\circ\Phi\circ\pi_\varphi^{-1}$ and hence a unitary $U:H_\varphi\to H_\psi$ defined by extension of $U(\pi_\varphi(x)\Omega_\Phi):=\Phi(x)\Omega_\psi$ for all $x\in\m$. It follows by definition that $\Phi\circ\pi_\varphi^{-1}=Ad_U$ intertwines $\beta$ and $\gamma$.
 \end{proof}

  Thus, the problem of classifying \en-semigroups up to conjugacy reduces to classifying \en-semigroups on von Neumann algebras in standard form up to spatial conjugacy. In what follows we will assume that all our \en-semigroups acts on  von Neumann algebras in standard form, which we say  \en-semigroups acting standardly. This allows us to use the following result of Araki.
  
 \begin{thm}\cite{Araki modular} \label{Araki's modular theorem}
    Let $\m$ be a von Neumann algebra with cyclic and separating vectors $\Omega_1$ and $\Omega_2$. If $J_1$ and $J_2$ are the corresponding modular conjugations then the *-automorphism $Ad_{J_1J_2}|\m\to\m$ is inner.
 \end{thm}

 \begin{defn} Let $\alpha$ be an \en-semigroup acting standardly on $\m \subseteq B(H)$ with vacuum vector $\Omega$. A unit for $\alpha$ is a strongly continuous semigroup $T=\{T_t:t \geq 0\}$ of operators in $B(H)$ such that $T_0=1$ and $T_tx=\alpha_t(x)T_t$ for all $t\geq0$, $x\in\factor$. Denote the collection of units by $\mathcal{U}_\alpha$.  \end{defn}

 It will follow from Theorem \ref{prod sys complete invariant} that the collection of units is an invariant for an \en-semigroup. When the vacuum state is invariant under $\alpha$, that is $\ip{\alpha_t(m)\Omega}{\Omega}=\ip{m\Omega}{\Omega}$ for all $t\geq 0$, $m\in \m$, there exists a unit $S_t$, which is the semigroup of isometries determined by $S_t x \Omega:=\alpha_t(x) \Omega$.  We call $\{S_t:t\geq 0\}$ the \emph{canonical $\Omega-$unit} associated to $\alpha$. When $\m$ is a II$_1$ factor, the trace is an invariant state and the associated canonical unit is an invariant under conjugacy. 

 A gauge cocycle for $\alpha$ is a cocycle $\{U_t: t\geq 0\}$, which satisfies the locality  condition $U_t\in \alpha_t(\m)' \cap \m$ for all $t\geq 0$. Under the multiplication $(UV)_t:=U_tV_t$, the collection of all gauge cocycles forms a group, denoted by $G(\alpha)$, called the gauge group of $\alpha$. $G(\alpha)$ is an invariant of $\alpha$ under cocycle  conjugacy. 

\begin{lem}\label{ONB}
Let $\alpha$ be an \en-semigroup on a factor $\m\subseteq B(H)$ in standard form. Then there exists a family  of isometries $\{U_i(t):i \in \Iset\}\subseteq B(H)$ satisfying
 \begin{itemize}
  \item[(i)] $\sum_{i\in \Iset} U_i(t)U_i(t)^*=1,$ 
  \item[(ii)] $\alpha_t(x) = \sum_{i\in \Iset} U_i(t)x U_i(t)^*$ for all $x \in \m$,
 \end{itemize} where the convergence in (i) and (ii) is in $\sigma-$weak topology. When $\m$ is a type III factor the indexing set $\Iset$ is singleton and otherwise $\Iset=\mathbb{N}$
\end{lem}

\begin{proof} We refer to proposition 2.1.1,
\cite{Arveson book} when $\m$ is type I factor, and proposition 3.2, \cite{alevras} when $\m$ is type II$_1$ factor. Proof of the case when $\m$ is a type II$_\infty$ factor is similar to the case of type II$_1$. 
When $\m$ is a type III factor, $H$ can be considered as a left module over $\m$ with respect to the identity action and also with $x\cdot \xi=\alpha_t(x)\xi$ for $\xi \in  H$.  Since a separable non-zero module over a type III factor is unique up to isomorphism, the existence of a  unitary $U_t\in B(H)$ satisfying $\alpha_t(x)= U_tx U_t^*$ is guaranteed. 
\end{proof} 

\begin{rems}
It is not clear whether we can choose the family of unitaries $(U_t)_{t\geq 0}$, describing  \en-semigroups on type III factor in the above Lemma, satisfying semigroup property $U_{s+t}=U_s U_t$ for all $s,t \geq 0$. \end{rems}

\begin{prop}
Let $\alpha$ and $\beta$ be two \en-semigroups on factors $\m_1$ and $\m_2$ respectively.Then there exists a unique \en-semigroup $\alpha\otimes \beta$ on $\m_1 \ot \m_2$ satisfying $$ (\alpha_t\otimes \beta_t) (m_1 \otimes m_2) = \alpha_t(m_1) \ot \beta_t(m_2)~~\forall m_1 \in \m_1, m_2 \in \m_2, t \geq 0.$$
\end{prop}

\begin{proof} 
For each $t \geq 0$, thanks to Lemma \ref{ONB}, choose isometries $\{U_i(t): i \in \Iset\}$ and $\{V_j(t): j \in \Jset\}  $, satisfying (i) and (ii) for $\alpha$ and $\beta$ respectively. Now the endomorphism $\alpha_t \otimes \beta_t$ is implemented by the family of isometries $\{U_i(t) \otimes V_j(t): i\in \Iset, j \in \Jset\}$. 
\end{proof}

We end this section by defining the basic examples of \en-semigroups on the type I$_\infty$ and hyperfinite II$_1$ factors. We recall the definitions of exotic type III examples on type I factors, and ask the reader to see relevant references for more details.

For a complex separable Hilbert space $K$, let $\Gamma_s(K):=\bigoplus_{n=0}^\infty{K^{\vee n}}$ be the the symmetric Fock space over $K$, i.e. the sum of symmetric tensor powers of $K$, and define the exponential vectors $\varepsilon(u):=\oplus_{n=0}^\infty{\frac{u\opower{n}}{\sqrt{n!}}}$, for each $u\in K$, and the vacuum vector is $\varepsilon(0)$. The exponential vectors are linearly independent and total in $\Gamma_s(K)$. The well-known isomorphism between $\Gamma_s(K_1)\ot\Gamma_s(K_2)\to \Gamma_s(K_1\op K_2)$, is given by the extension of $\varepsilon(u)\ot\varepsilon(v)\mapsto \varepsilon(u+v)$. Define the Weyl operator by
 $$W_0(u)\varepsilon(v):=e^{-\frac{1}{2}\norm{u}^2-\ip{u}{v}}\varepsilon(u+v) \qquad (u,v\in K),$$ which extends to a unitary operator on $\Gamma_s(K)$. $\{W_0(u):u \in K\}$ satisfies the well-known Weyl commutation relations
  $$W_0(u)W_0(v) = e^{-i\im \ip{u}{v} }W_0(u+v)~~ \forall u,v \in K.$$  For a unitary operator $U$ between 
$K_1$ and $K_2$, define the second quantisation $\Gamma(U)$ by $$ \Gamma(U)(\varepsilon(u))= \varepsilon(Uu)~u\in K,$$ which again extends to a unitary operator between $\Gamma_s(K_1)$ and $\Gamma_s(K_2)$.  We can also define the second quantisation for anti-unitaries as well in the same way, first by defining on exponential vectors but then extending anti-linearly.

Let $\kil$ be a real Hilbert space. 
Let $K= L^2((0,\infty),\kil^\Comp)$ denote the square integrable functions taking values in $\kil^\Comp$. Throughout this paper we denote by $(T_t)_{t\geq 0}$ the right shift semigroup on $L^2((0,\infty),\kil^\Comp)$ (or its restriction to $L^2((0,\infty),\kil)$) defined by 
\begin{eqnarray*}(T_tf)(s) & = & 0, \quad s<t,\\
& = & f(s-t), \quad s \geq t,
\end{eqnarray*} 
for $f \in K$.
The \textit{CCR flow} of index $\dim \kil$ is the \en-semigroup $\theta=\{\theta_t:t\geq 0\}$ acting on $B(\Gamma_s(L^2((0,\infty),\kil^\Comp)))$ 
defined by the extension of $$\theta_t(W_0(f)):=W_0(S_tf), ~~ f \in L^2((0,\infty),\kil^\Comp).$$ The CCR flow of index $n$ is cocycle conjugate to the CCR flow of index $m$ if and only if $m=n$ (see \cite{Arveson book}).

Generalised CCR flows are defined in \cite{genccr} as follows. Let $\{T^1_t\}$ and $\{T^2_t\}$ be two $C_0$-semigroups acting on a real Hilbert space $G$.
We say that $\{T^1_t\}$ is a perturbation of $\{T^2_t\}$, if they satisfy,
\begin{enumerate}
\item[(i)] ${T^1_t}^*T^2_t =1.$
\item[(ii)] $T^1_t -T^2_t$ is a Hilbert Schmidt operator.
\end{enumerate}
\medskip
Given a perturbation $\{T^1_t\}$ of $\{T^2_t\}$, there exists a unique \en-semigroup $\theta=\{\theta_t: t\geq 0\}$ on 
$B(\Gamma_s(G^\Comp))$ defined and extended by 
$$\alpha_t(W_0(x+iy))=W_0(T^1_tx+iT^2_ty),\quad x,y\in G.$$ $\theta$
is called as the \textit{generalised CCR flow} associated with the pair $(\{T^1_t\},\{T^2_t\})$. 

Toeplitz CAR flows are introduced in \cite{toepcar}. Let $K$ be a complex Hilbert space.  
We denote by $\A(K)$ the CAR algebra over $K$, which is the universal $C^*$-algebra generated by 
$\{a(x): x \in K\}$, where $x \mapsto a(x)$ is an antilinear map satisfying the CAR relations:  
\begin{eqnarray*}
a(x)a(y) +a(y)a(x)&  = & 0, \\
a(x)a(y)^* +a(y)^*a(x) & = & \ip{x}{y}1, 
\end{eqnarray*} for all $x,y \in K$. 
Since $\A(K)$ is known to be simple, and any set of operators satisfying the CAR relations generates a $C^*$-algebra 
canonically isomorphic to $\A(K)$. 
The \textit{quasi-free state} $\omega_A$ on $\A(K)$, associated with a positive contraction $A \in B(K)$, 
is the state determined by its $2n$-point function as  
$$\omega_A(a(x_n) \cdots a(x_1)a(y_1)^* \cdots a(y_m)^* ) = \delta_{n,m} \det (\langle x_i, Ay_j\rangle) ,$$ 
where $\det(\cdot)$ denotes the determinant of a matrix (see Chapter 13, \cite{Arveson book}).
Given a positive contraction, it is a fact that such a state always exists and is uniquely determined by the above relation. 
We denote by $(H_A, \pi_A, \Omega_A)$ the GNS triple associated with 
a quasi-free state $\omega_A$ on $\A(K)$, and set $\m_A:=\pi_A(\A(K))''$. 

Now let $K= L^2((0,\infty),\kil^\Comp)$ and $A\in B(K)$ be a positive contraction satisfying $\tr(A-A^2)<\infty$ and  $T_t^*AT_t=A$ for all $t$. Then $\m_A$ is a type I factor and there exists a unique \en-semigroups $\theta=\{\theta_t:t\geq 0\}$ on  $\m_A$, determined by 
$$\theta_t(\pi_A(a(f)))=\pi_A(a(T_tf)),\quad \forall f\in K.$$
$\theta$ is called the \textit{Toeplitz CAR flow} associated with $A$ (see Chapter 13, \cite{Arveson book}).

Next we recall the examples of \en-semigroups on hyperfinite type II$_1$ factors (see \cite{Pow Pow}, \cite{alevras} and \cite{MS} for discussions on these examples). For a real Hilbert space $K$, let
$\Gamma_a(K^\Comp):=\bigoplus_{n=0}^\infty{(K^\Comp)^{\wedge n}}$ be the the antisymmetric Fock space over $K^\Comp$, i.e. the sum of antisymmetric tensor powers of $K$.
For any $f\in K^\Comp$ the Fermionic creation operator $a^*(f)$ is the bounded operator defined by the linear extension of
$$a^*(f)\xi=\left\{\begin{array}{ll} f & \hbox{if}~\xi=\Omega, \\ f\wedge \xi &  \hbox{if}~ \xi\perp\Omega, \end{array}\right.$$ where $\Omega$ is the vacuum vector ($1$ in the 0-particle space $\Comp$), and 
$f\wedge\xi$ is the antisymmetric tensor product. The annihilation operator is defined by $a(f)=a^*(f)^*$.
The unital $C^*$-algebra $Cl(K)$ generated by the self-adjoint elements $$\{u(f)=(a(f)+a^*(f))/\sqrt{2}:~f\in K\}$$ is the Clifford algebra over $K$. The vacuum $\Omega$ is cyclic and defines a tracial state for $Cl(K)$, so the weak completion yields a II$_1$ factor; in fact it is the hyperfinite II$_1$ factor $\mathcal{R}$.

Now if $K=L^2((0,\infty), \kil)$, where $\kil$ is a separable real Hilbert space with dimension $n\in \overline{\Nat}$ as mentioned before, then there exists a unique  \en-semigroup on $\mathcal{R}$ by extension of $$\alpha^n_t(u(f_1)\cdots u(f_k))=u(T_tf_1)\cdots u(T_tf_k), ~~ f_1\cdots f_k \in K,$$ called the \textit{Clifford flow} of rank $n$.
The von Neumann algebra generated by the even products $$\r_e=\{u(f_1)u(f_2)\cdots u(f_{2n}): f_i \in L^2((0,\infty), \kil),~n \in \Nat\}$$ is also isomorphic to the hyperfinite II$_1$ factor. The restriction of the Clifford flow $\alpha^n$ of rank $n$ to this subfactor is called as the even Clifford flow of rank $n$. In \cite{alevras}, an index for \en-semigroups on II$_1$ factors is defined, and it is shown that the index of a Clifford flow (or an even Clifford flow) equals to its rank.

In \cite{MS}, it was shown that a Clifford flow (respectively an even Clifford flow) of rank $n$ is cocycle conjugate to a Clifford flow (respectively an even Clifford flow) of rank $m$ if and only if $m=n$. This was proven by using a theory of $C^*-$semiflows, which are defined as follows.
Let $\alpha$ be an \en-semigroup on a II$_1$ factor. For each $t\geq0$ let $\semiflowalg_\alpha(t):=\alpha_t(\m)'\cap\m$. These algebras form an increasing filtration. Define the inductive limit $C^*$-algebra $\mathcal{A}_\alpha:=\overline{\bigcup_{t\geq0} \mathcal{A}_\alpha(t)}^{\norm{\cdot}}$, together with a semigroup of *-endomorphisms $\alpha|_{\semiflowalg_\alpha}$. This is called the $C^*$-semiflow corresponding to $\alpha$. Since this is a subalgebra of $\m$ there is a canonical trace on $\semiflowalg_\alpha$ which we denote by $\tau_{\alpha}$. Two cocycle conjugate \en-semigroups have isomorphic (in the obvious sense of the word) $\tau$-semiflows (see \cite{MS}, Section 9).  We will be using the following fact in section \ref{IIinfty}. Again see \cite{MS} for the details of the proof.

\begin{prop}\label{semiflow}
Any two Clifford flows (or even Clifford flows) are cocycle conjugate if and only if they are conjugate if and only if they have isomorphic $\tau$-semiflows if and only if they have the same index. 
\end{prop}

 \section{The Coupling Index}

In this section we extend the definition of the coupling index from \cite{MS} to \en-semigroups on an arbitrary factor, and show that it is a cocycle conjugacy invariant. Let $\alpha$ be an \en-semigroup on a factor $\m$ with cyclic and separating vector $\Omega$ and let $J_\Omega$ be the modular conjugation associated to the vector $\Omega$ by Tomita-Takesaki theory. We can define a complementary \en-semigroup on $\m'$ by setting
 $$\alpha^{J_\Omega}_t(x')=J_\Omega\alpha_t(J_\Omega x'J_\Omega)J_\Omega\qquad (x'\in\m').$$
 When the context is clear (i.e. for fixed $\Omega$) we sometimes denote $\alpha^{J_\Omega}$ simply by $\alpha^\Omega$, or $\alpha'$.

 \begin{prop}\label{doubly intertwining unitary shizzle}
 Let $\m$ and $\n$ be von Neumann algebras acting standardly with respective cyclic and separating vectors $\Omega_1\in H_1$, $\Omega_2\in H_2$. If the \en-semigroups $\alpha$ on $\m$, and $\beta$ on $\n$ are cocycle conjugate, then $\alpha^{J_{\Omega_1}}$ and $\beta^{J_{\Omega_2}}$ are cocycle conjugate. Moreover, if $\alpha$ and $\beta$ are conjugate, then $\alpha^{J_{\Omega_1}}$ and $\beta^{J_{\Omega_2}}$ are spatially conjugate and the implementing unitary can be chosen so that it also intertwines $\alpha$ and $\beta$.
 \end{prop}

\begin{proof}
 If $\alpha$ is conjugate to $\beta$ via the isomorphism $\theta$, then by Lemma \ref{spatial conjugacy lemma} there is a unitary $U:H_1\to H_2$ implementing the conjugacy and a cyclic separating vector $\Omega_\theta\in H_1$ with $Ux\Omega_\theta=\theta(x)\Omega_2$ for all $x\in\m$. It is clear that $UJ_{\Omega_\theta}=J_{\Omega_2}U$ and hence
 $$\beta^{J_{\Omega_2}}_t(x)=J_{\Omega_2}U\alpha_t(U^*J_{\Omega_2}xJ_{\Omega_2}U)U^*J_{\Omega_2}=UJ_{\Omega_\theta}\alpha_t(J_{\Omega_\theta}U^*xUJ_{\Omega_\theta})J_{\Omega_\theta}U^*$$
 for all $x\in\n'$. It follows from Theorem \ref{Araki's modular theorem} that the *-isomorphism $\m'\to\m'$, $x\mapsto J_{\Omega_\theta}J_{\Omega_1}xJ_{\Omega_1}J_{\Omega_\theta}$ is inner, so let $V\in\m'$ be the implementing unitary. Then the right hand side becomes
 $$UV J_{\Omega_1}\alpha_t(J_{\Omega_1}V^*U^*xUV J_{\Omega_1})J_{\Omega_1}V^*U^*=UV\alpha^{J_{\Omega_1}}_t((UV)^*xUV)(UV)^*.$$
 So $\alpha^{J_{\Omega_1}}$ and $\beta^{J_{\Omega_2}}$ are spatially conjugate and, since $V\in\m'$, we also have
 $$UV\alpha_t((UV)^*xUV)(UV)^*=U\alpha_t(U^*xU)U^*=\beta_t(x)$$
 for all $x\in\n$.
 
 For cocycle conjugacy we may assume that $\m=\n$, $\Omega=\Omega_1=\Omega_2$ and $\{U_t:t\geq0\}$ is an $\alpha-$cocycle such that $\beta_t(\cdot)=U_t\alpha_t(\cdot)U_t^*$. For any $t \geq 0$, let $V_t=J_\Omega U_tJ_\Omega$, then $V_t \in \m'$ and $V_t$ satisfies $$V_{s+t}=J_\Omega U_{s+t}J_\Omega=J_\Omega U_sJ_\Omega J_\Omega\alpha_s(U_t)J_\Omega= J_\Omega U_sJ_\Omega\alpha_s^{J_\Omega}(J_\Omega U_tJ_\Omega)=V_s\alpha_s^{J_\Omega}(V_t).$$ So $\{V_t:t\geq0\}$ forms an $\alpha^{J_\Omega}$-cocycle. We also have \begin{align*}\beta_t^{J_\Omega}(m') &=J_\Omega\beta_t(J_\Omega m'J_\Omega)J_\Omega =J_\Omega U_t\alpha_t(J_\Omega m'J_\Omega)U_t^*J_\Omega \\ &=(J_\Omega U_tJ_\Omega)(J_\Omega \alpha_t(J_\Omega m'J_\Omega)J_\Omega)(J_\Omega U_t^*J_\Omega) =V_t\alpha_t^{J_\Omega}(m')V_t^*, \end{align*} for all $m'\in \m'$. \end{proof}

\begin{defn} Let $\alpha$ be an \en-semigroup on the von Neumann algebra $\m$ acting standardly on $H$ with cyclic and separating vector $\Omega$. An $\Omega$-\munit or $\Omega$-multi-unit for the \en-semigroup $\alpha$ is a strongly continuous semigroup of bounded operators $(T_t)_{t\geq0}$ in $B(H)$ satisfying
$$T_tx=\left\{\begin{array}{ll}{\alpha}_t(x)T_t & \text{if } x\in \m, \\ {\alpha}_t^{J_\Omega}(x)T_t& \text{if } x\in\m',\end{array}\right.$$ 
together with $T_0=1$. That is, an $\Omega$-multi-unit is an $\Omega$-unit for both $\alpha$ and ${\alpha}^{J_\varphi}$. Denote the collection of $\Omega$-\munits for $\alpha$ by $\munitset{\alpha}^\Omega$ (or by $\munitset{\alpha}^\varphi$, if $\varphi$ is the faithful normal state associated with $\Omega$). We say that $\alpha$ is multi-spatial, or $\mu$-spatial, if it admits a multi-unit.
\end{defn}

\begin{ex}\label{multi spatial example}
 An \en-semigroup $\alpha$ on a \twoone factor $\m$ is automatically multi-spatial. Indeed, the canonical unit with respect to the trace  
is a   \munit for $\alpha$. On the other hand a type III \en-semigroup on a type I factor is not multi-spatial, as it follows from Example \ref{B(H) example}. In Section \ref{IIinfty}, we provide examples of \en-semigroups on type II$_\infty$ factors which are not multi-spatial. 
\end{ex}

 The following proposition gives a large number of multi-spatial examples, for which \en-semigroups on \twoone factors are a special case. 
 
 \begin{prop}\label{invariant state}
  Let $\alpha$ be an \en-semigroup acting standardly on a factor $\m$ with cyclic and separating vector $\Omega$, and $\varphi$ be the faithful normal state associated with $\Omega$.  Then the following are equivalent:
 \begin{rlist}
   \item $\varphi$ is an invariant state for $(\m,\alpha)$, and the corresponding canonical unit $(S_t)_{t\geq 0}$ is a $\Omega$-\munit.
   \item $\varphi$ is an invariant state for $(\m,\alpha)$, and for all $t\geq0$, the canonical unit $(S_t)_{t\geq 0}$ and modular conjugation $J$ satisfy $S_t=JS_tJ$.
   \item For all $t\geq0$, $s\in\Real$ the modular group satisfies $\alpha_t=\sigma_{-s}^{\Omega}\circ\alpha_t\circ\sigma_s^\Omega$.
  \end{rlist}
   \end{prop}

 \begin{proof}
  (i)$\Rightarrow$(ii). For all $m'\in\m'$, $t\geq0$, we have
  $$S_tm'\Omega=\alpha^J_t(m')\Omega=J\alpha_t(Jm'J)\Omega=JS_tJm'\Omega,$$
  so $S_t=JS_tJ$ for all $t\geq0$.
  
  (ii)$\Rightarrow$(i). For all $m'\in\m'$, $t\geq0$, we have
  $$S_tm'=JS_tJm'J^2=J\alpha_t(Jm'J)S_tJ=\alpha^J_t(m')S_t.$$
  
  (ii)$\Rightarrow$(iii). For all $t\geq0$, $m\in\m$,
  $$\Delta^{1/2} S_t m\Omega=\Delta^{1/2}\alpha_t(m)\Omega=J\alpha_t(m^*)\Omega=JS_tm^*\Omega=S_tJm^*\Omega=S_t\Delta^{1/2}m\Omega,$$
  so $\Delta^{1/2} S_t\supseteq S_t\Delta^{1/2}$. Thus  $$\sigma^\Omega_s\circ\alpha_t(m)\Omega=\Delta^{is}S_tm\Omega=S_t\Delta^{is}m\Omega=\alpha_t\circ\sigma^\Omega_s(m)\Omega.$$
  (See e.g. \cite{Conway} Section X.)
 
  (iii)$\Rightarrow$(ii). From the commutation relation we see that, for all $t\geq0$, the state $\varphi\circ\alpha_t$ satisfies the KMS condition for $\sigma^\Omega$. Thus, by uniqueness, $\varphi\circ\alpha_t=\varphi$ for all $t\geq0$. It also follows from the commutation relation that $\Delta^{is}S_t=S_t\Delta^{is}$ for all $s\in\Real$, thus we can infer that $\Delta^{1/2}S_t\supseteq\Delta^{1/2}S_t$ and, by the *-preserving property of $\alpha$, $JS_t=S_tJ$ for all $t\geq0$.
 \end{proof}

 For multi-spatial \en-semigroups we can introduce a numerical index. The first step towards this is the following Lemma.

 \begin{lem} Let $\m$ be a factor acting standardly and $X$ and $Y$ be $\Omega$-\munits for the \en-semigroup $\alpha$ on the factor $\m$. Then $X_t^*Y_t=e^{\lambda t}1$ for some constant $\lambda\in \Comp$.\end{lem}

 \begin{proof} 
 By routine arguments, it is easy to see that $X_t^*Y_t\in(\m\cup\m')'= \Comp 1$. 
Further the complex valued function $f$ satisfying  $X_t^*Y_t = f(t)1$ is continuous and satisfies $f(s+t)=f(s)f(t)$. Since $f(0)=1$ we have $f(t)=e^{\lambda t}$ for some $\lambda\in\Comp$.   \end{proof}

 Thus, for a multi-spatial \en-semigroup $\alpha$, we can define a covariance function $c:\mathcal{U}_{\alpha,\alpha'}^\Omega\times\mathcal{U}_{\alpha,\alpha'}^\Omega\to\Comp$ by $X_t^*Y_t=e^{c(X,Y) t}1$ for all $t\in\Rplus$. Since the covariance function is conditionally positive definite (see Proposition 2.5.2 of \cite{Arveson book}) the assignment
 $$ \ip{f}{g}\mapsto \sum_{X,Y\in \mathcal{U}_{\alpha,\alpha'}^\Omega}{c(X,Y)\overline{f(X)}g(Y)} $$
 defines a positive semidefinite form on the space of finitely supported functions $f:\mathcal{U}^\Omega_{\alpha,\alpha'}\to\Comp$ satisfying $\sum_{X\in \mathcal{U}^\Omega_{\alpha,\alpha'}}{f(X)}=0$. Hence, if this space is nonempty, we may quotient and complete to obtain a Hilbert space $H(\mathcal{U}^\Omega_{\alpha,\alpha'})$. The following proposition shows that this space is independent of $\Omega$ and is a cocycle conjugacy invariant for $\alpha$. 

 \begin{prop}\label{coupling ind invariant}
  Let $\alpha$ and $\beta$ be cocycle conjugate \en-semigroups on respective factors $\m$ and $\n$ acting standardly with cyclic and separating vectors $\Omega_1$ and $\Omega_2$. Then there is a bijection $\munitset{\alpha}^{\Omega_1}\to\munitset{\beta}^{\Omega_2}$ which preserves the covariance function. In particular, if one \en-semigroup is multi-spatial, then so is the other, and we have $H(\mathcal{U}^{\Omega_1}_{\alpha,\alpha'})\cong H(\mathcal{U}^{\Omega_2}_{\beta,\beta'})$.
 \end{prop}

 \begin{proof}
  If $\alpha$ and $\beta$ are conjugate, then the unitary $UV$ constructed in the proof of Proposition \ref{doubly intertwining unitary shizzle} induces a bijection $\munitset{\alpha}^{\Omega_1}\to\munitset{\beta}^{\Omega_2}$, $S\mapsto UVS(UV)^*$, since it intertwines  both the pairs of \en-semigroups $(\alpha,\beta)$ and their associated complementary \en-semigroups $(\alpha', \beta')$. Moreover, for any $S,T\in\munitset{\alpha}^{\Omega_1}$ and $t\geq0$ we have $S_t^*T_t\in\Comp1_\m$ so
  $$e^{tc(UVS(UV)^*,UVT(UV)^*)}1_\n=UVS_t^*T_tV^*U^*=e^{tc(S,T)}1_\n.$$
  If $\m=\n$, $\Omega_1=\Omega_2=\Omega$ and $\alpha=\beta^U$, for an $\alpha$-cocycle $U=(U_t)_{t\geq0}$, then one checks that the map $S\mapsto UJ_\Omega UJ_\Omega S$ gives a bijection $\munitset{\alpha}^\Omega\to\munitset{\beta}^\Omega$, and clearly this preserves the covariance function. The proposition follows.
 \end{proof}
 
 \begin{defn} For a multi-spatial \en-semigroup $\alpha$, define the coupling index $\Ind_c(\alpha)$ as the cardinal $\dim H(\mathcal{U}^\varphi_{\alpha,\alpha'})$ for some faithful normal state $\varphi$ on $\m$.
 \end{defn}

 \begin{lem} For \en-semgiroups $\alpha$, $\beta$ on $\m$ and $\n$ respectively we have $$\Ind_c(\alpha\otimes\beta)\geq \Ind_c(\alpha)+\Ind_c(\beta).$$\end{lem}
 
 \begin{proof}
  Pick faithful normal states $\varphi$ and $\psi$ on $\m$ and $\n$ and note that every pair $(X^\alpha$, $X^\beta)\in\munitset{\alpha}^\varphi\times\munitset{\beta}^\psi$ gives a $(\varphi\ot\psi)$-\munit $X^\alpha\otimes X^\beta$ for $\alpha\otimes\beta$. As
  $$({X_t^\alpha}\otimes {X_t^\beta})^*({Y_t}^\alpha\otimes {Y_t^\beta})=e^{(c(X^\alpha,Y^\alpha)+c(X^\beta,Y^\beta))t}1$$
  there exists an isometry $$H(\munitset{\alpha}^\varphi)\oplus H(\munitset{\beta}^\psi)\hookrightarrow H(\mathcal{U}_{\alpha\otimes\beta,(\alpha\otimes\beta)'}^{\varphi\ot\psi})$$ (see \cite{Arveson book} Lemma 3.7.5). 
 \end{proof}

\begin{prop}\label{extn} Let $\alpha$ be an \en-semigroup on a factor $\m\subseteq B(H)$ with cyclic and separating vector $\Omega$. If there exists an \en-semigroup $\sigma$ on $B(H)$ satisfying 
 $$\sigma_t(x)=\left\{\begin{array}{ll}  {\alpha}_t(x) & \text{if }x\in\m, \\  {\alpha}_t^{J_\Omega}(x) & \text{if }x\in \m', \end{array}\right.\qquad\text{for all }t\geq0,$$
 then $\Ind_c(\alpha)$ is equal to the Powers-Arveson index of $\sigma$. If $\beta$ be another \en-semigroup on a factor $\n\subseteq B(K)$ with cyclic and separating vector $\Omega_1$, which is  cocycle conjugate to $\alpha$,  then there exists an \en-semigroup $\theta$ on $B(K)$ extending both $\beta$ and $\beta^{J_{\Omega_1}}$.

In this case we say $\alpha$ is canonically extendable, a property which is invariant under cocycle conjugacy, and $\sigma$ is called as the canonical extension.

\end{prop}

\begin{proof} 
 Clearly if $T$ is a unit for $\sigma$ then it is an $\Omega$-\munit for $\alpha$. Conversely if $T$ is an $\Omega$-\munit for $\alpha$, since multiplication is separately ultraweak continuous, and  
by the ultraweak continuity of $\sigma$, it follows that
$T$ is a unit for $\sigma$. Thus $\mathcal{U}^\Omega_{\alpha,\alpha'}=\mathcal{U}_\sigma$ and the induced covariance function on $\mathcal{U}_\sigma$ is precisely that of \cite{Arveson book}, Section 2.5.
 
 Suppose $U:H\mapsto K$ be unitary and $(U_t)_{t \geq 0}$ be a unitary cocyle for $\alpha$ satisfying $\beta_t=  Ad_U Ad_{U_t}\alpha_t Ad_{U^*}$ for all $t\geq 0$. By Proposition \ref{doubly intertwining unitary shizzle} there exists a $V:H\mapsto K$ implementing the conjugacy of both $(Ad_{U_t} \alpha_t)_{t\geq 0}$ and $\beta$ as well as the conjugacy of $\left(\left( Ad_{U_t}\alpha_t \right)^{J_\Omega}\right)_{t\geq 0}(= \left(Ad_{J_{\Omega} U_t J_{\Omega}}\alpha_t^{J_\Omega}\right)_{t\geq 0})$ and $\beta^{J_{\Omega_1}}$. Now $$\theta_t = Ad_{V} Ad_{U_t} Ad_{J_\Omega U_t J_\Omega}\sigma_t Ad_{V^*}$$  provides the canonical extension for $\beta$.
   \end{proof}

\begin{rem} All \en-semigroups on type I factors are canonically extendable. All our known examples of \en-semigroups on II$_1$ factors (and type II$_\infty$ factors)  are not canonically extendable (see [MS]). It is an interesting open problem to construct a canonically extendable \en-semigroup on a II$_1$ factor.  On type III factors we know both extendable and non-extendable examples.
\end{rem}

 The following is a fundamental example for the theory.

 \begin{ex}\label{B(H) example}
  Let $\m=B(H)$ and $\ol{H}$ be the dual space of $H$, with an anti-isomorphism $\xi \mapsto \ol{\xi}$ from $H\mapsto \ol{H}$. Consider the standard representation $\pi:\m\to B({H}\otimes\ol{H}))$, defined by linear extension of $\pi(X)({\xi}\ot\ol{\eta})={X\xi}\ot\ol{\eta}$, with cyclic and separating vector $\Omega=\sum_{n=1}^\infty{\frac{1}{n}{e_n}\ot\ol{e_n}}$, where $\{e_n\}_{n=1}^\infty$ is an orthonormal basis for $H$. We claim that the corresponding modular conjugation is given by $$J{\xi}\ot\ol{\eta}={\eta}\ot\ol{\xi}.$$
  To see this, define an operator $\Delta^{1/2}_0$ on $\pi(\m)\Omega$ by
  $$\Delta^{1/2}_0\left(\sum_{n=1}^\infty{\frac{1}{n} Qe_n  \ot \ol{e_n}}\right) =\sum_{n=1}^\infty{\frac{1}{n} e_n \ot \ol{Q^*e_n}}$$
  for all $Q\in B(H)$, and note that if $Qe_n=\sum_{m=1}^\infty{Q_{mn}e_m}$, then
 \begin{align*}
  \ip{\pi(Q)\Omega}{\Delta_0^{1/2}\pi(Q)\Omega}&=\sum_{n,m=1}^\infty \frac{1}{nm}\left< \left(\sum_{i=1}^\infty Q_{in}{e_i}\right)\ot \ol{e_n},    {e_m} \ot \left(\sum_{j=1}^\infty \ol{Q_{mj}}\ol{e_j}\right) \right>\\
   &=\sum_{n,m=1}^\infty\frac{1}{mn}|Q_{mn}|^2\geq0 
  \end{align*}
  hence $\Delta_0^{1/2}$ extends to a closed, densely defined, positive operator $\Delta^{1/2}$. Clearly $J\Delta^{1/2}\pi(Q)\Omega=\pi(Q)^*\Omega$, so $\Delta^{1/2}$ is the modular operator and $J$ is the modular conjugation.
 
  If $X$ is an operator on $H$  then let $\ol{X}$ be the operator on the dual space defined by $\ol{X}\ol{\eta}=\ol{X\eta}$, so $J(X\ot 1)J=1\ot\ol{X}$. Let $\alpha$ be an \en-semigroup on $\m$ and denote by $\beta$ the conjugate semigroup $\pi\circ\alpha\circ\pi^{-1}$ on $\pi(\m)$. Then we have
  $$\beta^J_t(1\ot \ol{X})=J(\alpha_t(X)\ot 1)J=1\ot\ol{\alpha_t(X)}.$$
  Thus the dual \en-semigroup $\beta^J$ is conjugate to an \en-semigroup $\ol{\alpha}$ on $B(\ol{H})$ given by $\ol{\alpha}_t(\ol{X})=\ol{\alpha_t(X)}$.  Clearly the \en-semigroup $\alpha\ot\ol{\alpha}$ extends both $\beta$ and $\beta^J$, so by \cite{Arveson book} $\alpha$ is multi-spatial if and only if it is spatial, in which case the index is $\Ind_c(\alpha)=\Ind(\alpha\ot\ol{\alpha})=2\Ind(\alpha)$; its coupling index is twice its Powers-Arveson index.
 \end{ex}

 \begin{rem}
  The previous example suggests that a better definition for coupling index would be half the dimension of $H_{\alpha,\alpha'}$. However, by historical accident, our first paper on coupling index considered only \en-semigroups on type \twoone factors. We will not attempt to redact the original definition, since we do not have any reason to believe an arbitrary \en-semigroup must have an even coupling index, though constructing an example with odd coupling index is an open problem.
 \end{rem}

 \section{Product systems}\label{pdct}
 
Product systems of Hilbert modules have been extensively studied by M. Skiede, with contributions from Bhat and others (see \cite{skeide} for an elaborate discussion). In \cite{alevras}, Alevras associated intertwiner spaces as a product system of Hilbert modules for an \en-semigroup on a II$_1$ factor, and showed that two \en-semigroups are cocycle conjugate if and only if the associated product systems are isomorphic (for other, similar results see \cite{skeide holyoke}). 

In this section we generalize Alevras' association of product systems for \en-semigroups on general factors. We prove that they form a complete invariant with respect to the equivalence of cocycle conjugacy. 
When the proofs are exactly similar, we do not give full details of the proofs, but ask the reader to refer to \cite{alevras}.

For our purposes, we use a different but equivalent definition of Hilbert modules from \cite{BMSS}. All our modules are von Neumann modules. 

\begin{defn} For a von Neumann algebra $\m \subseteq B(H)$ in standard form, a Hilbert von Neumann $\m-$module is defined as a weakly closed (equivalently, strongly closed) subspace $E\subseteq B(H)$ satisfying $EE^*E\subseteq E$ and $E^*E=M$. $E$ is a Hilbert von Neumann $\m-\m-$bimodule if further 
$\m \subseteq \left[EE^*\right]$. 
\end{defn}

$E^*E$ is the von Neumann algebra acting on the right and $\left[EE^*\right]$ is the collection of adjointable operators acting on the left. The inner product is $\ip{x}{y}=x^*y$ (see \cite{skeide}, Part I, Chapter 3  or \cite{BMSS} for details).  Further in this definition non-degeneracy is automatic.

The following version of Riez' Lemma is useful for our purposes. For a proof the reader can refer either to \cite{skeide} or to \cite{BMSS}, Proposition 1.7. In this lemma $E_1$, $E$ can be taken as only right von Neumann modules. By a submodule of a von Neumann $\m-$module $E$, we mean a strongly closed subspace $E_1$ satisfying $E_1E^*E=E_1\m\subseteq E_1$.

\begin{lem}\label{riez}
If $E_1$ is an $\m$-submodule of a Hilbert von Neumann $\m-$module $E$ and $E_1 \neq E$. Then there exists a non-zero $y \in E$ such that $y^*x=0$ for all $x \in E_1$. 
\end{lem}

The following definition of isomorphism between von Neumann modules is equivalent to the usual definition of Hilbert modules (see `if' implication in Lemma 2.5, \cite{BMSS}).

\begin{defn} Two von Neumann modules $E\subseteq B(H)$ and $F \subseteq B(K)$ over von Neumann algebras $\m\subseteq B(H)$ and $\n\subseteq B(K)$ in standard form are isomorphic if there exist unitaries $U_1, U_2: H \mapsto K$ satisfying $$U_i\m U_i^*=\n~ i=1,2;~ U_1^*U_2 \in \m'; ~U_1EU_2^*=F.$$ We say $E$ and $F$ are isomorphic through $(U_1, U_2)$.
\end{defn}

\begin{defn}
A complete orthonormal basis for an $\m-\m-$bimodule $E$  is a countable family of isometries with orthogonal ranges $\{S_i; i\in \Iset\}\subseteq E$ satisfying $\sum_{i \in \Iset} S_iS_i^*=1$ in strong topology. When it exists we say the bimodule admits a complete orthonormal basis.
\end{defn}

The countable family of isometries in the above definition can possibly be a single unitary. Notice when an $\m-\m-$bimodule admits a complete orthonormal basis $\{S_i;i\in \Iset\}$, any element $T\in E$ can be written as $T= \sum_{i\in \Iset} S_iS_i^* T$ in strong topology. Since $S_i^* T \in \m$, we have $[S_i m_i: i\in \Nat, m_i \in \m]=E$.

For an \en-semigroup $\alpha$ on a factor $\m \subseteq B(H)$ in standard form with cyclic and separating vector $\Omega$, define $$E_t^\alpha = \{T \in B(H): \alpha_t(x)T=Tx~\forall x \in \m\}.$$

\begin{prop} Let $\alpha_t$ be an \en-semigroup acting standardly on a factor $\m$. For each $t\geq 0$ $E_t^\alpha$ is an $\m'-\m'-$bimodule. Further $\left[EE^*\right]=\alpha_t(\m)'$. 
\end{prop}

\begin{proof} It is easy to verify from the intertwining property of elements in $E$, exactly as in  proposition 3.1 \cite{alevras}, that $E^*E\subseteq \m'$, $EE^*\subseteq \alpha_t(\m)'$ and $E\m'\subseteq E$, $\alpha_t(E) E\subseteq E$. Let $S_t$ be either any one of  the isometries or a unitary provided by Lemma \ref{ONB}. Notice that $S_t \in E$, and any $m'\in \m'$ can be expressed as $S_t^*(S_t m)\in E^*E$. Hence $E^*E=\m'$. Again it is exactly similar, as in  proposition 3.1 \cite{alevras}, to verify that $\left[EE^*\right]$ forms a two sided ideal in $\alpha_t(\m)'$.
\end{proof}

For two Hilbert von Neumann $\m-\m-$bimodules $E_1$ and $E_2$ the internal tensor product is defined by $$E_1 \odot E_2 =\left[xy: x \in E_1, y\in E_2\right].$$ This definition coincides with the definition of internal tensor product of Hilbert bimodules in \cite{skeide} (see \cite{BMSS}).
 
\begin{defn}\label{prdct} Let $\m\subseteq B(H)$ be a factor in standard form. A concrete product system of $\m-\m-$Hilbert bimodules is a one parameter family $\{E_t:t\geq 0\}$ of $\m-\m-$von Neumann bimodules, such that 
$\E=\{(t, T_t): t \in (0,\infty), T_t \in E_t\}$ is a standard Borel subset of $(0,\infty)\times B(H)$ (equipped with the product structure coming from $(0,\infty)$ and strong topology in $B(H)$),
satisfying
$E_s \odot E_t =E_{s+t}~~\forall s,t\geq 0.$

Further, for each $t\geq 0$,  the von Neumann $\m-\m-$module $E_t$ admits a complete orthonormal basis $\{S^i_t: i \in \Iset\}$, such that 
$\{(t,S^i_t):t\in (0,\infty)\}$ is a Borel subset of $\E$ for each $i \in \Iset$.
\end{defn}

The definition for isomorphism between product systems of von Neumann modules given below is equivalent to the one in \cite{alevras} and elsewhere. The condition  ${U^i_s}^*U^j_t\in \m'~\forall s,t\geq 0, i,j=1,2,$ is to ensure all unitaries implement the same isomorphism on $\m$.

\begin{defn}
Two product systems $\{E_t:t\geq 0\}$ and $\{F_t:t\geq 0\}$ over $\m$ and $\n$ are isomorphic if there exists families of unitaries $\{U_t^1: t\geq 0\}$ and $\{U^2_t:t\geq 0\}$ such that $E_t$ and $F_t$ are isomorphic through $(U_t^1, U_t^2)$ for each $t\geq 0$, satisfying ${U^i_s}^*U^j_t\in \m'~\forall s,t\geq 0, i,j=1,2,$ and  \begin{align}\label{product morphism}  \left(U^1_sT_s{U_s^2}^*\right)\left(U^1_tT_t{U_t^2}^*\right) & = U_{s+t}T_sT_tU_{s+t}^*~~\forall T_s \in E_s, T_t\in E_t~s,t\geq 0.\end{align} Further $(t,T_t)) \mapsto  (t, U^1_t T_t{U^2_t}^*)$ is a Borel isomorphism from $\E$ onto $\F = \{(t, T_t): t \in (0,\infty), T_t \in F_t\}$.
\end{defn}

\begin{thm}
Let $\alpha$ be an \en-semigroup acting standardly on a factor $\m\subseteq B(H)$. Then $\{E^\alpha_t:t\geq 0\}$ is a concrete product system of Hilbert modules. 
\end{thm}

\begin{proof} Let $s,t\geq 0$. Clearly $ST\in E^\alpha_{s+t}$ for $S \in E^\alpha_s$ and $T \in E^\alpha_t$. Choose $\{U_i(s):i \in \Iset\}$ and $\{U_j(s):j \in \Jset\}$  as in Lemma \ref{ONB}. Then any $A\in E_{s+t}$ can be expressed in strong limit as $A=\sum_{i,j} U_i(s)U_j(t)U_i(s)^*U_j(t)^* A.$ But  $U_i(s)\in E_s^\alpha$ and $U_i(t)U_i(s)^*U_j(t)^* A \in E_t$. It follows  that  $$E_s^\alpha\odot E_t^\alpha=E^\alpha_{s+t}, ~~\forall s,t\geq 0.$$

It can be proven, by using exactly same arguments as in \cite{alevras} (see arguments before Lemma 3.8), that $\E^\alpha =\{(t,T_t): t \in (0,\infty), T_t\in E^\alpha_t\}$ is a Borel subset of $(0,\infty)\times B(H)$. 
$E^\alpha_t$ admits a complete orthonormal basis, thanks to Lemma \ref{ONB}.  When $\m$ is not a type III factor, both $\alpha_s$ and $\alpha_t$ are representations of infinite multiplicity. Hence they are both unitarily equivalent to the amplification of the standard representation on $H\ot K$ with $dim(K)=\infty$. When $\m$ is a type III factor all representations are unitarily equivalent to the standard representation. So for any $s, t\in (0,\infty)$, $\alpha_s$ and $\alpha_t$ are unitarily equivalent considered as representations of $\m$. Now the measurability of the complete orthonormal basis, as in Definition \ref{prdct},  can be proven exactly in the same manner as in \cite{alevras} (see Lemma 3.8 and Corollary 3.9).
\end{proof}


The proof of the following proposition also implies that an \en-semigroup acting standardly on $\m \subseteq B(H)$ has an extension to an \en-semigroup on $B(H)$.

\begin{prop}\label{samepdct}
Let $\{E_t:t\geq 0\}$ be a  concrete product system of bimodules over $\m'$, contained in $B(H)$. Then there exists a unique \en-semigroup $\alpha$ on  $\m$  whose associated product system is $\{E_t:t\geq 0\}$.
\end{prop} 

\begin{proof} Suppose two unital $*-$endomorphism $\alpha_t$ and $\beta_t$ on $\m$ have $E_t$  as their intertwiner space, then $\alpha_t(x)T\xi=Tx\xi=\beta_t(x)T\xi$ for all $T \in E_t, \xi \in H, x \in \m$. Since each $E_t$ admits a complete orthonormal basis we have $\left[E_tH\right]=H$ for all $t\geq 0$. Hence $\alpha_t=\beta_t$. So the product system uniquely determines \en-semigroup on $\m \subseteq B(H)$.

Now choose a complete orthonormal basis $\{S^i_t:i \in \Iset\}$ in $E_t$ for each $t\geq 0$. Define a unital $*-$endomorphism $$\theta_t(X) = \sum_{i\in \Iset} S_t^i X (S^i_t)^* \qquad X \in B(H).$$ We have $\theta_t(X)S^i_t =S_t^iX$ for all $X \in B(H)$, $i \in \Iset$. 
When  $m \in \m$,  
since $\{S^i_t: i \in \Iset\}$ generates $E_t$ as a right $\m'$ module, we have $\theta_t(m)T= T m$ for all $T \in E_t$. In particular $\theta_t(m)$ commutes with $\left[E_tE_t^*\right]$ which contains $\m'$. Hence $\theta_t$ leaves $\m$ invariant, and we denote the restriction of $\theta_t$ to $\m$ by $\alpha_t$.  The intertwiner space of $\alpha_t$ is a right Hilbert $\m'$ module, containing $E_t$ as a submodule. Suppose $T\in B(H)$ satisfies $T^*S=0$ for all $S \in E_t$, then $T^*=0$, since $[E_tH]=H$. Hence $T=0$, and thanks to Lemma \ref{riez}, the intertwiner space of $\alpha_t$ is exactly $E_t$. 

Notice that $\alpha_t$ does not depend on the particular choice of the complete orthonormal basis, since $E_t$ determines $\alpha_t$ uniquely.  As we can choose the complete orthonormal basis in a measurable way, we have obtained a measurable family $\alpha=\{\alpha_t:t\geq 0\}$ of unital normal $*-$endomorphisms. The set of isometries $\{S^i_sS_t^j: i,j \in \Iset\}$ provides a complete orthonormal basis in $E_{s+t}$. So  
$$ \alpha_{s+t}(m)  = \sum_{i,j\in \Iset} S_s^i S^j_t m (S^j_t)^* (S^i_s)^*
  = \sum_{i\in \Iset} S_s^i \alpha_t(m) (S^i_s)^*
 = \alpha_s(\alpha_t(m))$$ for all $m \in \m$, 
and hence $\alpha$ is indeed a semigroup. Now Section 2.3 of \cite{Arveson book} implies $\alpha$ is strongly continuous, and $\alpha$ is an \en-semigroup on $\m$. 
\end{proof}


The following theorem asserts that the isomorphism class of the product system of Hilbert modules is well-defined up to cocycle conjugacy and that it is a complete invariant. In particular it does not depend on the particular standard representation.

\begin{thm}\label{prod sys complete invariant}
Let $\alpha$ and $\beta$ be two \en-semigroups acting standardly on factors $\m$ and $\n$ respectively.  Then $\alpha$ and $\beta$ are cocycle conjugate if and only if the associated product system of Hilbert modules are isomorphic.
\end{thm}

\begin{proof}
If $\alpha$ and $\beta$ are conjugate \en-semigroups with conjugacy implemented by unitary $U$, then the isomorphism between $E^\alpha_t$ and $E^\beta_t$ is implemented by $(U,U)$. When $\beta$ is a cocycle perturbation of $\alpha$ by a cocycle $\{U_t: t\geq 0\}$, then the isomorphism between $E^\alpha_t$ and $E^\beta_t$ is implemented by $(U_t, 1)$. The measurability follows from of the strong continuity of $\{U_t:t\geq 0\}$.

Conversely assume $\{E^\alpha_t:t\geq 0\}$ and $\{E^\beta_t:t\geq 0\}$ are isomorphic as product systems of Hilbert modules. If the isomorphism is implemented by a single unitary $(U, U)$ for all $t\geq 0$, then $Ad_{U}\circ\alpha\circ Ad_{U^*}$ and $\beta$ have $\{E^\beta_t:t\geq 0\}$ as their product systems. Thanks to the uniqueness assured by Lemma 
\ref{samepdct} $\alpha$ and $\beta$ are conjugate \en-semigroups. 
Now by considering  a conjugate \en-semigroup, we may assume that both $\{E^\alpha_t:t\geq 0\}$ and $\{E^\beta_t:t\geq 0\}$ are contained in $B(H)$ and that the isomorphism is implemented by $(U^1_t, U^2_t)$ with $U^1_t, U^2_t \in \m$. But $U^1_t T (U^2_t)^* = U^1_t \alpha_t ((U^2_t)^* T$ for any $T \in E^\alpha_t$. 
So we assume, without loss of generality, that the isomorphism is implemented by $(U_t^1, 1 )$ for $U_t \in \m  ~~\forall t\geq 0$.   Since $(U_t,1)$ satisfies equation \ref{product morphism}, we have $$U_s\alpha_s(U_t)T_sT_t=U_sT_sU_tT_t= U_{s+t}T_s T_t, ~~\forall T_s\in E_s, T_t\in E_t~ s,t,\geq 0.$$ Since $[E_{s}H]=H=[E_{t}H]$ we conclude that $\{U_t:t\geq \}$ satisfies the cocycle relation for $\alpha$. The measurability assumption implies the strong continuity of $\{U_t:t\geq 0\}$ (see section 2.3, \cite{Arveson book}). Now again appealing to lemma \ref{samepdct}, we conclude that $\beta$ is $\{U_t\}-$perturbation of $\alpha$.
\end{proof}

\begin{rem}\label{type I pdct} 
Suppose $\alpha$ is an \en-semigroup acting standardly on a type I factor, for instance as in example \ref{B(H) example}. Let $\{H_t:t\geq 0\}$ be the concrete Arveson product system associated with $\alpha$, contained in $\m$. Then the product system of Hilbert modules associated with  $\alpha$ is described by $$E_t^\alpha = \left[\{T_t \ot S: T_t \in H_t, S \in B(\ol{H})\}\right].$$ This can be verified by Lemma \ref{riez}. 
It is clear for two \en-semigroups, the associated product system of Hilbert modules are isomorphic if and only if the respective product system of Hilbert spaces are isomorphic.  
\end{rem}

 \section{Super-product systems}\label{SPS}
 
In this section we associate a super-product system to each \en-semigroup on a factor and show that it is a cocycle conjugacy invariant. The super-product system of an \en-semigroup is a generalisation of Arveson's product system for an \en-semigroups on a type I factor. However, in contrast to type I factors, the super-product system for an \en-semigroup on an arbitrary factor is not a complete invariant, as shown for the case of type III factors, in Section \ref{CCR section}. These were originally defined in \cite{MS}, but the idea was already known to some experts.
 
 \begin{defn}
 \label{superproductsystem}
 A super-product system of Hilbert spaces is a one parameter family
 of separable Hilbert spaces $\{H_t: t > 0 \}$, together with isometries $$U_{s,t} : H_s \otimes H_t ~\mapsto H_{s+t}~ \mbox{for}~ s, t \in (0,\infty),$$
 satisfying the following two axioms of associativity and measurability.

 \medskip
 \noindent(i) (Associativity) For any $s_1, s_2, s_3 \in (0,\infty)$
 $$U_{s_1, s_2 + s_3}( 1_{H_{s_1}} \otimes U_{s_2 ,
 s_3})=  U_{s_1+ s_2 , s_3}( U_{s_1 ,
 s_2} \otimes 1_{H_{s_3}}).$$

 \noindent (ii) (Measurability)  The space $\H=\{(t, \xi_t): t \in (0,\infty), \xi_t \in H_t\}$ is equipped with a structure of standard Borel space that is compatible with the projection $p:\H\mapsto (0,\infty)$ given by $p((t, \xi_t)=t$, tensor products and the inner products (see 3.1.2, \cite{Arveson book}). 
\end{defn}

A super-product system is an Arveson product system if the isometries $U_{s,t}$ are unitaries and further the axiom of local triviality is satisfied, which is equivalent to the existence of countable measurable total set of sections. 
 
 \begin{prop} Let $\m\subset B(H)$ be a factor acting standardly with cyclic and separating vector $\Omega$ and $\alpha$ an \en-semigroup on $\m$. For each $t>0$, let
 $$H^{\alpha,\Omega}_t=E^\alpha_t \cap E^{\alpha^\Omega}_t=\{ X\in B(H):~\forall_{m\in\m}~Xm=\alpha_t(m)X, \forall_{m'\in\m'}~Xm'=\alpha^\Omega_t(m')X \}, $$
 then $H^{\alpha,\Omega}=\{H^{\alpha,\Omega}_t: t>0\}$ is a concrete super-product system with respect to the the family of isometries $U_{s,t}\left(X\ot Y \right)=XY$.
 \end{prop}
 
 \begin{proof} It is routine to verify that
$X^*Y\in (\m\cup\m')'=\Comp1$ for any $X, Y \in H^{\alpha,\Omega}_t$. Clearly each $H^{\alpha,\Omega}_t$ is closed under the operator norm, and this coincides with the norm induced by the inner-product $\ip{X}{Y}1:=X^*Y$, hence each $H^{\alpha,\Omega}_t$ is a Hilbert space with respect to this inner product. It is straightforward to check for $X\in H^{\alpha,\Omega}_s$, $Y\in H^{\alpha,\Omega}_t$, that $XY\in H^{\alpha,\Omega}_{s+t}$ and that the map $U_{s,t}(X \ot Y)=XY$ is an isometry. Since $H^{\alpha,\Omega}_t =E^\alpha_t\cap E^{\alpha'}_t$ and the measurability axiom follows from the measurability of the product system of Hilbert modules $(E^\alpha_t)_{t>0}$ and $(E^{\alpha'}_t)_{t>0}$.
 \end{proof}

\begin{defn} By an isomorphism between super product systems $(H^1_t, U^1_{s,t})$ and $(H^2_t, U^2_{s,t})$ we mean an isomorphism of Borel spaces $V:\H^1 \mapsto \H^2$ whose restriction to each fiber provides an unitary operator $V_t:H^1_t\mapsto H^2_t$ satisfying \begin{equation*}\label{prodiso}V_{s+t}U^1_{s,t}= U_{s,t}^2 (V_s \otimes V_t).\end{equation*} 
\end{defn}

 A priori, the super-product system appears to depend upon the chosen state $\Omega$. The following theorem shows that this is not the case.
 
 \begin{thm}\label{sps theorem}
  Let $\alpha$ and $\beta$ be \en-semigroups acting standardly on respective factors $\m$ and $\n$ with cyclic and separating vectors $\Omega_1$ and $\Omega_2$. If $\alpha$ and $\beta$ are cocycle conjugate then $H^{\alpha,\Omega_1}$ and $H^{\beta,\Omega_2}$ are spatially isomorphic.
 \end{thm}

 \begin{proof}
  We prove the theorem in three stages. First we show that for any two cyclic and separating vectors $\Omega_1$, $\Omega_2$, $H^{\alpha,\Omega_1}$ and $H^{\alpha,\Omega_2}$ are isomorphic. By Theorem \ref{Araki's modular theorem} there exists a unitary $V\in \m'$ such that $J_{\Omega_1}J_{\Omega_2}m'J_{\Omega_2}J_{\Omega_1}=Vm'V^*$ for any $m'\in\m'$. We claim the maps $H^{\alpha}_t\ni X\mapsto VXV^*$ give the required isomorphism. Indeed, $VXV^*$ is clearly an intertwiner for $\alpha$, and
  \begin{align*}VXV^*m'&=VX(V^*m'V)V^*=VJ_{\Omega_1}\alpha_t(J_{\Omega_1}V^*m'VJ_{\Omega_1})J_{\Omega_1}XV^*\\
   &=J_{\Omega_2}\alpha_t(J_{\Omega_2}m'J_{\Omega_2})J_{\Omega_2}VXV^*,
  \end{align*}
  so $VXV^*$ is an intertwiner for $\alpha^{\Omega_2}$. 
  
  Next, if $\alpha$ and $\beta$ are conjugate then, letting $U$ be the unitary implementing the conjugacy, we get an isomorphism $Ad_{J_\Omega UJ_\Omega U}:H^{\alpha,\Omega}_t\to H^{\beta,U\Omega}_t$.  Lastly, if $\beta$ is a cocycle perturbation of $\alpha$ by the cocycle $(U_t)_{t\geq0}$, then left multiplication by $J_\Omega U_tJ_\Omega U_t$ gives the required family of unitaries $H_t^{\alpha,\Omega}\to H_t^{\beta,\Omega}$.  
 \end{proof}
 
 We will thus talk freely of the (abstract) super-product system $\{H^{\alpha}_t,U_{s,t}\}$ for $\alpha$. As in the case of product systems, one can define units and an index for a super-product system (see \cite{Arveson book}, Section 3.6), and the index of the super-product system coincides with the coupling index of the \en-semigroup.
 
 \newcommand{\typeIsps}{\ensuremath{H_t\ot\ol{H}_t}}
 \newcommand{\tIsps}[1]{\ensuremath{H_t^{#1}\ot\ol{H}^{#1}_t}}
 \begin{ex}\label{H ot barH}
An Arveson product system $(\ol{H_t}, \ol{U}_{s,t})$ is dual to another Arveson product system if there exists a family of anti-unitaries $J_t: H_t\mapsto \ol{H_t}$ satisfying $$J_{s+t}U_{s,t}=\ol{U}_{s,t}\left( J_s \ot J_t\right).$$ There is no need to impose any measurability condition, since it is proved in \cite{lieb} that all measurable structures on an Arveson product system are isomorphic. Clearly dual of an Arveson product system is determined uniquely up to isomorphism, and double dual is isomorphic to the original Arveson system. Type I Arveson systems are self dual, but in general it is not clear whether every Arveson product system is isomorphic to its dual.

Suppose  $\alpha$ is an \en-semigroup acting standardly on a type I factor $\m$, as in example \ref{B(H) example}. If $\{H_t:t > 0\}$ is the Arveson's product system for $\alpha$,  then  the Arveson system of the dual \en-semigroup $\alpha^\Omega$ is dual to $\{H_t:t > 0\}$; indeed the map $T\mapsto J_{\Omega}TJ_{\Omega}$ provides the required dual isomorphism. We further have  $E^\alpha_t=[H_t \ot B(\ol{H})]$ and $E^{\alpha^\Omega}_t = [B(H)\ot \ol{H_t}]$. So the super-product system is 
$H^\alpha_t = \typeIsps$.
 \end{ex}
 
 \begin{defn}
Let $\alpha$ be an \en-semigroup on $\m\subseteq B(H)$, which is in a standard form. The super-product system $(H_t)_{t > 0}$, associated with $\alpha$ is said to be full if $H_tH=H$ for all $t > 0$.
\end{defn}

\begin{rem}
If a super product system $(H_t)_{t > 0}$ associated with an \en-semigroup $\alpha$ is full then it is a product system. In fact if $R \in H_{s+t}$ satisfying $R^*ST=0$ for all $S\in H_s, T\in H_t$, then $R^*\xi=0$ for all $\xi \in H$. But the converse is not true. The super product associated with free flows on $L(F_\infty)$  are one dimensional product systems and they are not full (see \cite{MS}). The tensor products of free flows with  \en-semigroups on a type I factor have associated super product systems which are infinite dimensional product systems but not full. 

The super product system associated with an \en-semigroup is full if and only if it is canonically extendable, as defined in Proposition \ref{extn}. If $\{U_i: i \in I\}$ be an orthonormal basis for $H_t$, then thanks to the fullness, $U_i$s are isometries with ranges summing to the whole of $H$. The \en-semigroup $\sigma_t(X)=\sum_{i \in I} U_iX U_i^*$ provides the canonical extension. The converse is well known (see \cite{Arveson book}).
\end{rem}

 \begin{prop}\label{tensorSPS}
Let $\alpha$ and $\beta$ be \en-semigroups on factors $\m_1$ and $\m_2$, then the super-product system for $\alpha\ot\beta$ is the tensor product of the super-product systems for $\alpha$ and $\beta$. \end{prop}

 \begin{proof}
Assume $\m_1\subseteq B(H_1)$ and $\m_2\subseteq B(H_2)$  are in standard form with respective cyclic and separating vectors $\Omega_1$ and $\Omega_2$, evaluating faithful normal states $\varphi_1$ on $\m_1$ and $\varphi_2$ on $\m_2$ respectively.
It is clear that $H^\alpha_t \ot H^\beta_t \subseteq H^{\alpha\ot \beta}_t$, we prove the other inclusion as follows.

Let $X \in H^{\alpha\ot \beta}_t$. 
Since the statement of the proposition is true for product systems, we assume that at least one of the super-product systems is not full.  Let $\Omega= \Omega_1 \ot \Omega_2$. Notice that any operator in the super-product system is determined by its value on the cyclic vector, through the relation $$X(m_1 \ot m_2) \Omega=\left(\alpha_t(m_1)\ot \beta_t(m_2)\right)X\Omega.$$ Suppose $X \in H^{\alpha\ot \beta}_t$ such that $X \perp H^\alpha_t \ot H^\beta_t$, then $X^*$ is zero on $H_t^\alpha H_1\ot H_t^\beta H_2$. This implies that the projection of $X\Omega$ onto $H_t^\alpha H_1\ot H_t^\beta H_2$ is zero.  Our strategy is to show that the projection of $X \Omega$ onto  $\left(H_t^\alpha H_1\ot H_t^\beta H_2\right)^\perp$ is also $0$, so it follows that $X=0$. 

Assume towards a contradiction that $0\neq X \Omega \in \left(H_t^\alpha H_1\ot H_t^\beta H_2\right)^\perp$. Let $\tilde{H}_t=\left(H_t^\alpha H_1\right)^\perp\subseteq H_1$.  Without loss of generality we assume that $H_1$ is not full, and hence that $\tilde{H}_t\neq \{0\}$ and that there exists a unit vector $\xi \in H_2$ such that $0\neq (1\ot P_\xi)X\Omega \in \tilde{H}_t \ot \Comp\xi$, where $P_\xi$ is the projection onto the one dimensional subspace spanned by $\xi$. (The other case can be dealt similarly).
 
Let $E_\xi:H_1\to H_1\ot H_2$ denote the isometry $\eta\mapsto\eta\ot\xi$, write $E^\xi:H_1\ot H_2\to H_1$ for its adjoint and note that $E_\xi E^\xi=1\ot P_\xi$. Define $T\in B(H_1)$ by $T=E^\xi XE_{\Omega_2}$, so that
 $$0\neq T\Omega_1 =E^\xi X\Omega\in\wt{H}_t.$$
 Then, for all $m_1,m_2\in\m$,
 \begin{align*}
  Tm_1m_2\Omega_1&=E^\xi X(m_1m_2\ot1)\Omega=E^\xi(\alpha_t(m_1)\ot1)X(m_2\ot1)\Omega\\
  &=\alpha_t(m_1)E^\xi X(m_2\ot1)\Omega=\alpha_t(m_1)Tm_2\Omega_1
 \end{align*}
 so that $Tm_1=\alpha_t(m_1)T$ and, similarly, $Tm_1'=\alpha_t'(m_1')T$ for all $m'\in\m'$. Thus $T\in H^\alpha_t$, contradicting $T\Omega_1\in\wt{H}_t$.
\end{proof}
 
 For a super-product system $H=\{H_t:t\geq 0\}$, we say $K=\{K_t\subseteq H_t:t\geq 0\}$ a super-product subsystem  if the product map of $H$ restricts to the product map of $K$.
 
 \begin{rem}\label{units tensors} If two super-product systems $H^1$ and $H^2$ can be embedded into product systems as super-product subsystems, then for any unit $u=\{u_t:t\geq0\}$ in $H^1\ot H^2$ there exists units $u^1=\{u_t^1:t\geq0\}$ in $H^1$ and $u^2=\{u_t^2:t\geq0\}$ in $H^2$ such that $u_t=u^1_t\ot u^2_t$. Since any unit in the super-product subsystem is also a unit for the bigger product system, this follows immediately from the corresponding statement in \cite{Arveson book} for product systems. 
 \end{rem}
 

 \section{\en-semigroups on II$_\infty$ factors}\label{IIinfty}
 
In this section we consider tensor products of \en-semigroups on a type I factor with \en-semigroups on type II$_1$ factors. This way we produce several (both countable and uncountable) families of  \en-semigroups on II$_\infty$ factors. Let $\r$ be the hyperfinite II$_1$ factor, and we always assume $\r \subseteq L^2(\r)$ with respect to the tracial state. Let $\r_\infty=B(H)\otimes \r$, then $\r_\infty$ is the hyperfinite II$_\infty$ factor.

In this section, $\alpha^{n}$ denotes either the Clifford flow or the even Clifford of rank $n $, with $n \in \overline{\Nat}$, and  when $n$ is fixed we just denote it by $\alpha$. The super-product systems of Clifford flows and even Clifford flows are computed in \cite{MS}. 
Set $$H_t^{e,n} =[\xi_1 \wedge\xi_2 \wedge \cdots \wedge \xi_{2m};~ \xi_1, \xi_2 \cdots \xi_{2m} \in L^2((0,t), \kil^\Comp), ~m \in \Nat_0],$$ for all $t\geq 0$, and $\dim(\kil)=n\in \overline{\Nat}$. We may write just $H^e_t$ in many instances when $n$ is fixed.  The super-product system of the Clifford flow (isomorphic to the super  product system of the even the Clifford flow) of rank $n$ is described by $H_t^{\alpha^{n}}\Omega =H^{e,n}_t$ for all $t\geq 0,$ where $\Omega\in L^2(\r)$ is the vacuum vector. The isometries $U_{s,t}:H_s^{\alpha^n}\Omega \otimes H_t^{\alpha^n}\Omega\mapsto H_{s+t}^{\alpha^n}\Omega$ are given by 
\begin{align*}  U_{s,t}((\xi_1 \wedge\xi_2 \wedge \cdots \wedge \xi_{2m}) \otimes (\eta_1 \wedge\eta_2  \wedge \cdots \wedge \eta_{2m'}) )&\\ =\xi_1 \wedge\xi_2  \wedge \cdots \wedge \xi_{2m}\wedge T_s \eta_1 \wedge T_s \eta_2\wedge  \cdots \wedge T_s \eta_{2m'}& \end{align*} where $\xi_1, \xi_2\cdots \xi_{2m} \in L^2(0,s)$, $\eta_1, \eta_2 \cdots \eta_{2m} \in L^2(0,t)$.  We consider different families of \en-semigroups of the form $\theta \ot \alpha$ on $\r_\infty$  by varying $\theta$ on $B(H)$.

\subsection{Tensoring with CCR flows}
Throughout this subsection, let $\theta^{m}=\{\theta^{m}_t:t\geq 0\}$ denote the CCR flow of index $m\in \overline{\Nat}$ on  $B(H^m)$, where $H^m=\Gamma_s(L^2(\Rplus,\kil^\Comp))$, $dim(\kil)=m$. The (Arveson) product system of Hilbert spaces associated with $\theta^m$ is the well-known exponential product system $\{H^m_t: t \geq 0\}$ of index $m$ which are described as follows:  $H^m_t = \Gamma_s(L^2((0,t),\kil^\Comp))$ with $dim(\kil)=m$ and the unitaries $U_{s,t}: H^m_s\ot H^m_t\mapsto H^m_{s+t}$ are the extensions of $\varepsilon(x)\ot \varepsilon(y)\mapsto \varepsilon(x+y)$.

\begin{thm}
$\theta^{m} \otimes \alpha^{n}$ is cocycle conjugate to $\theta^{p} \otimes \alpha^{q}$ if and only if $(m,n) =(p,q)$.
\end{thm}

\begin{proof} We prove the theorem in two steps. First we assume the cocycle conjugacy and prove that $m=n$. In the next step we assume $m=n$ and prove that the cocycle conjugacy implies $p=q$.

\emph{Step 1: Assume $\theta^{m} \otimes \alpha^{n}$ is cocycle conjugate to $\theta^{p} \otimes \alpha^{q}$. }

Thanks to Proposition \ref{tensorSPS}, the super-product system of $\theta^{m} \otimes \alpha^{n}$ is given by $(\tIsps{m})\otimes H^{e,n}_t$. Since the super-product system $H^{e,n}$ can be embedded into the product system corresponding to the CAR flow (on type I factor) of index $n$, thanks to Remark \ref{units tensors}, units in $(\tIsps{m}) \otimes H^{e,n}_t$ are of the form $u_t \otimes v_t$, with $u_t$ a unit for $(\tIsps{m})_{t\geq 0}$ and $v_t$ a unit for $(H^{e,n}_t)_{t\geq 0}$. But the super-product system $(H^{e,n}_t)_{t\geq 0}$ has only the canonical unit, as the unique unit up to a scalar (see Section 8, \cite{MS}). So by comparing the coupling index we get $m=p$. 

\emph{Step 2: Take $\theta^{m}=\theta^{p}=\theta$,  $H^m_t=H_t$ and assume  $\theta \otimes \alpha^{n}$
is cocycle conjugate to $\theta \otimes \alpha^{q}$.}

Set $\m=B(H)$. We assume $\r_\infty=\m \ot \r\subseteq B(H \ot \ol{H})\ot B(L^2(\r))$ is in standard form, by identifying $\m$ with $B(H)\ot 1$, and (without loss of generality) that both the semigroups act on the same algebra.
Suppose that there exists a $\theta\ot\alpha^n$-cocycle $U$ in $\r_\infty$ and a unitary $V\in B(H  \ot \ol{H}\ot L^2(\r))$ such that
$$ \theta_t\ot \alpha_t^p= Ad_{VU_t}\circ(\theta_t\ot\alpha_t^n)\circ Ad_{V^*} \qquad \forall t\geq0 .$$
Let $(S^n_t)_{t\geq 0}, (S^p_t)_{t \geq 0}$ be the canonical units in $B(L^2(\m))$ for $\alpha^n$ and $\alpha^p$ respectively. Notice that $\theta$ and its complementary \en-semigroup $\theta'$ extends to $\theta\ot \theta'$ on $B(H \ot  \ol{H})$, and the super-product system $\typeIsps$ is the product system of Hilbert spaces associated with $\theta\otimes \theta'$. The multi-units of $\theta$ are just the units of $\typeIsps$ in the sense of Arveson (see \cite{Arveson book}). 

Let $u_t\otimes S_t^n$ be a unit for $\typeIsps \ot H^{e,n}_t $, with $u_t$ a unit for $\typeIsps$. 
Let $J=J_1 \otimes J_2$, with $J_1, J_2$ modular conjugation for $\m$ and $\r$, with respect to vacuum vectors $\Omega_1$ and $\Omega_2$ respectively. 
Let $U_t'=JU_tJ$. Then $(VU_t'U_t (u_t\otimes S^n_t) V^*)_{t\geq 0}$ is a unit for $(\typeIsps \otimes H^{e,p}_t)_{t\geq 0}$, which is of the form $(v_t \ot  S^p_t)_{t\geq 0}$, for some unit $(v_t)_{t\geq 0}$ for $\typeIsps$. Since the (left)action of $(U_t'U_t)_{t\geq 0}$ and $Ad_V$ on the units preserves the covariance function, the map $u\mapsto v$ also preserves the covariance function.  So there is an induced automorphism of $(\mathcal{U}, c)$ (see Definition 3.74 and Section 3.8, \cite{Arveson book}),  where $\mathcal{U}$ is the collection of units for $\typeIsps$ and $c$ is the corresponding covariance function.  As proved in Section 3.8, \cite{Arveson book}, this automorphism is given by a gauge cocycle of $\theta\ot \theta'$; so there exists a gauge cocycle $(W_t)_{t\geq 0}$ of $\theta\ot \theta'$ satisfying \begin{align}\label{unit products} VU_t'U_t (u_t\ot S^n_t)V^* & =W_tu_t\ot S^p_t \qquad \forall u_t \in \mathcal{U}.\end{align} 
It is also clear that $$(U_t'U_t)^*V^*(v_t\ot S^p_t)V=W_t^*v_t\ot  S^n_t \qquad \forall v_t \in \mathcal{U}.$$
For every choice of units $u_1, \cdots , u_n$ in $\typeIsps$, $t_1,\cdots, t_n\in \Rplus$ satisfying $t_1+\cdots +t_n=t$,
we have
\begin{align*} VU_t'U_t ( ( u_{t_1}\cdots  u_{n} ) \otimes S^n_t )V^*& = (VU_{t_1}'U_{t_1} (u_{t_1}\ot S^n_{t_1} )V^* )\cdots  (VU'_{t_n}U_{t_n} (u_{t_n}\ot S^n_{t_n} )V^* )\\ & =  (W_{t_1}u_{t_1}\ot S^p_{t_1} )\cdots  (W_{t_n}u_{t_n}\ot S^p_{t_n} )\\ &=  W_t u_{t_1}\cdots u_{n} \otimes S^p_t ,
 \end{align*}
 where we have used the properties of $(U_t)_{t\geq 0}$ and $(W_t)_{t\geq 0}$ being cocycles, $(u_t)_{t\geq 0},$ $(S^n_t)_{t\geq 0}$ and $(S^p_t)_{t\geq 0}$ being units, and equation (\ref{unit products}). Since the product system of a CCR flow is generated by units, (and by a similar argument) we get
 \begin{align}\label{TR} VU_t'U_t (T\ot S^n_t )V^* & =W_tT\ot  S^p_t;\qquad (U_t'U_t)^*V^* (R\ot S^p_t )V=W_t^*R\ot  S^n_t,  \end{align} for all $T, R \in \typeIsps$.

 Now, for any $X\in\theta_t(\m)'\cap\m$, $T\in\typeIsps$, we have
\begin{align*} VU_t (X\ot 1)U_t^*V^* (T\ot S_t^p) & = VU_t'U_t(X\ot 1)(U_t'U_t)^*V^* (T\ot S_t^p)V V^* \\
 &=VU_t'U_t(X\ot1) (W_t^*T\ot S_t^n)V^*\\ 
 &= VU_t'U_t (XW_t^*T\ot S_t^n) V^*
  \\ &=W_tXW_t^*T\ot S_t^p , \end{align*}
where we have used equation (\ref{TR}) and the fact that $X W_t^*T \in \typeIsps$.  
It follows that for any $\xi\in H \ot\ol{H}$ and $m'\in \r'\cap B(L^2(\r))$
\begin{align*}
 VU_t(X\ot 1)U_t^*V^* (T\xi\ot m'\Omega_2)&=(1\ot m')VU_t(X\ot 1)U_t^*V^*(T\xi\ot S_t^p\Omega_2)\\
  &=(1\ot m') (W_tXW_t^* T\xi \ot S_t^p\Omega_2)
  \\ &= (W_tXW_t^*\ot 1)(T\xi\ot m'\Omega_2).
\end{align*}
Since the product system $\typeIsps$ is full and $\Omega_2$ is cyclic for $\m'$, we have
\begin{align}\label{gaugecocycle} Ad_{VU_t} (X\otimes 1 )= & Ad_{W_t}(X)\otimes 1 \qquad \forall X \in \theta_t(\m)'\cap \m.\end{align}
 Since $U_t \in \r_\infty$ and $Ad_V$ is an automorphism of $\r_\infty$ it follows a fortiori that $Ad_{W_t}(X)\in \m$ for all $X \in \theta_t(\m)'\cap \m$. Now, from the explicit description of gauge cocycles given in Section 9.8 of \cite{Arveson book}, it follows that $W_t$ is product of gauge cocycles of $\theta$ and $\theta^\prime$, and we assume, without loss of generality, that $(W_t)_{t\geq 0}\subseteq  \m$ is a gauge cocycle of $\theta$.

Now we consider the $C^*-$semiflows associated with these \en-semigroups.  For $i=n,p$, let
$$ \mathcal{C}^i_t =  ((\theta_t\ot \alpha^i_t) (\r_\infty ) )' \cap\r_\infty; \qquad \mathcal{A}^i_t= \alpha^i_t (\r )' \cap\r \qquad t\geq0,$$ 
$$ \mathcal{C}^i = \overline{\bigcup_{t\geq 0} (((\theta_t\ot \alpha_t^i) (\r_\infty ))' \cap\r_\infty )}^{\|\cdot\|}; \qquad \mathcal{A}^i= \overline{\bigcup_{t\geq 0} (\alpha^i_t (\r )' \cap\r )}^{\|\cdot\|} .$$ 
The inductive limit $\phi$ of the maps $\phi_t:=Ad_{VU_t}|\mathcal{C}^n_t\to\mathcal{C}^p_t$ provides an isomorphism between $\mathcal{C}^n$ and $\mathcal{C}^p$ intertwining the $C^*-$semiflows.

By equation (\ref{gaugecocycle}), we have that
\begin{align*} Ad_{(W_t^*\ot 1)VU_t}(1\ot Y)(X\ot 1)&=  (W_t^*\ot 1)VU_t(1\ot Y) (X\ot 1)U_t^*V^*(W_t\ot 1)
  \\&=(X\ot 1) Ad_{(W_t^*\ot1)VU_t}(1\ot Y)
\end{align*} 
for all $X\in\theta_t(\m)\cap\m$ and $Y\in\r$. Hence, for all $Y\in\mathcal{A}^n_t$,
$$Ad_{(W_t^*\ot 1)VU_t}(1\ot Y)\in ((\theta_t(\m)'\cap\m)\ot 1)'\cap\mathcal{C}^p_t=1\ot\mathcal{A}^p_t,$$
where the latter equality follows from the distributive property of tensors.
It follows that for each $t\geq0$, $\phi_t$ restricts to a map from $1\ot\mathcal{A}^n_t$ to $1\ot\mathcal{A}^p_t$, and hence $\phi$ restricts
to an isomorphism intertwining the $C^*$-semiflows for $\alpha^n$ and $\alpha^p$.

We claim that $\phi$ intertwines the tracial states on the $\mathcal{A}^i_t$ induced by the canonical trace on $\r$.
Indeed, by \cite{alevras} Proposition 2.9 each $\mathcal{A}^i_t$ is a \twoone factor and hence the maps $\phi_t$
intertwine the induced traces on each of the corresponding subalgebras - the statement follows by taking inductive limits.
In the terminology of \cite{MS}, $\alpha^n$ and $\alpha^p$ have isomorphic $\tau$-semiflows, and hence by Proposition \ref{semiflow} $n=p$.
\end{proof}

\subsection{Tensoring with Generalised CCR flows} Throughout this subsection, we denote by $\theta=\{\theta_t:t\geq 0\}$ a generalised CCR flow associated with a pair $\left(\{T^1_t\}_{t\geq 0}, \{T^2_t\}_{t\geq 0}\right)$, where $\{T^1_t: t\geq 0\}$ and $\{T^2_t: t\geq 0\}$ are two $C_0-$semigroups which are perturbations of one another. In our examples we assume the semigroup $\{T^1_t: t\geq 0\}$ is the right shift on $L^2(0,\infty)$ with index $1$.

In \cite{genccr}, local algebras associated with product systems were used to distinguish generalised CCR flows given by off-white noises with spectral density converging to $1$ at infinity. Here we define and use local algebras associated with super product systems to study \en-semigroups on the hyperfinite II$_\infty$ factor, given by tensor products of such generalised CCR flows with $\alpha$ (either a Clifford flow or an even Clifford flow with a fixed index). 

Let $H=(H_t, U_{s,t})$ be any super product system. Fix an arbitrary $a >0$. The local algebra $\A^H(I)$ associated with the super product system $H$ for any interval $I=(s,t)\subseteq [0,a]$ is defined by 
$$\A^H(I)=U_I^a\left(\Comp 1_{H_s} \ot B(H_{t-s}) \ot \Comp 1_{H_{a-t}}\right)(U_I^a)^*,$$ where $U_I^a$ is the canonical isometry $U_I^a:H_s\ot H_{t-s} \ot H_{a-t}\mapsto H_a$ determined uniquely by the associativity axiom. Here we consider $\A^H(I)$ as a von Neumann subalgebra of $B(P_I^aH_a)$, where $P_I^a=U_I^a(U_I^a)^*$.  

For an elementary open set $\open_N=\cup_{n=1}^N(s_n, t_n)$, denote the projection $P_{\open_N}^a= U_{\open_N}^a (U^a_{\open_N})^*$ where \begin{align}\label{Uopen} U^a_{\open_N}:\bigotimes_{n=1}^N H_{t_n-s_n} \otimes \bigotimes_{n=0}^N H_{s_{n+1}-t_n}\mapsto H_a\end{align} is the canonical isometry uniquely determined by the associativity axiom of the super product system. (Here we have set $t_0=0$ and $s_{N+1}=a$.) We just write $U_{\open_N}$ for $U^a_{\open_N}$ and $P_{\open_n}$ for $P^a_{\open_N}$ when $a$ is unambiguously fixed. For $1\leq k \leq N$, if we denote $ I_k=(s_k, t_k), ~\open_{k]} =\cup_{n=1}^{k-1}(s_n, t_n), ~ \open_{[k} =\cup_{n=k+1}^{N}(s_n-t_k, t_n-t_k), $  then using the associativity axiom, it is not difficult to verify that $$U^a_{\open_N} =  U^a_{I_k}\left(U^{s_k}_{\open_{k]}}\otimes 1_{H_{t_k-s_k}}\otimes U^{a-t_k}_{\open_{[k}}\right).$$ Using this we see for $x \in B(H_{t_k-s_k})$, \begin{align*} P^a_{\open_N} U_{I_k}^a\left( 1_{H_{s_k}} \ot x \ot 1_{H_{a-t_k}}\right)(U_I^a)^*  =  & U^a_{\open_N} \left( 1_{H_{s_k}} \ot x \ot 1_{H_{a-t_k}}\right) (U^a_{\open_N})^*\\   =  & U_{I_k}^a  \left( 1_{H_{s_k}} \ot x \ot 1_{H_{a-t_k}}\right)  (U_I^a)^*  P^a_{\open_N}, \end{align*} and hence $P^a_{\open_N} \in \A_{s_k, t_k}'$ for all $1\leq k\leq N$.

For a general open set  $\open \subseteq [0,a]$ with $\open= \cup_{n=1}^\infty I_n$ as disjoint union of intervals, define $$P_\open= \bigwedge_{n=1}^\infty P_{\open_n},$$ where
$\open_n=\cup_{k=1}^n I_k$ an increasing sequence of elementary open sets. 
$P_\open$ does not depend on the choice of the intervals or the elementary open sets $\{\open_n\}_{n=1}^\infty$, since $P_{\open_n}\leq P_{\open_m}$ if the elementary sets satisfies  $\open_m \subseteq \open_n$. (Caution: The relation $P_{\open_2}\leq P_{\open_1}$ does not hold in general for arbitrary elementary sets satisfying  $\open_1 \subseteq \open_2$; but it holds for sets in this collection, since the interval components of the elementary open subset is a subcollection of the interval components of the bigger elementary open set.)   Every $P_{\open_m}$ commutes with $\A(I_n)$ if $I_n \subseteq \open_m$. So $P_{\open}$ also commutes with $\A(I_n)$. Define $$\A^H(\open) = \bigvee_{n=1}^\infty P_{\open}\A^H(I_n),$$ the von Neumann algebra generated by $\{P_{\open}\A^H(I_n)\}_{n=1}^\infty$ in $B(P_\open H_a)$.

If the family $(V_t)_{t\geq 0}$ provides an isomorphism between two super product systems $(H_t, U_{s,t})$ and $(H_t', U_{s,t}')$, then $Ad(V_a)$ provides an isomorphism between $\A^H(\open)$ and $\A^{H'}(\open)$. Hence the family of von Neumann algebras $\{\A^H(\open): \open \subseteq [0,a]\}$ is an invariant for the super product system $(H_t, U_{s,t})$, hence for the associated \en-semigroup.

\begin{lem} \label{alg-typeI} Let $H$ be a super product system and $\open =\bigcup_{n=1}^\infty I_n\subseteq [0,a]$ is an open set for mutually disjoint open intervals $I_n=(s_n,t_n)$. Then
\begin{itemize}
\item [$(1)$] If $H$ is spatial, then $\A^H(\open)$ has a direct summand that is a type I$_\infty$ factor.  Further if $t_n<s_{n+1}$, then  $\A^H(\open)$ is a type I factor.
\item [$(2)$] If $H=H^e$ be the super product system associated with a Clifford flow of any fixed index, then $\A^H(\open)$ is a type I$_\infty$ factor for any open set $\open\subseteq [0,a]$. 
\end{itemize} 
\end{lem}

\begin{proof} 
Let $(S_t)_{t>0}$ be a unit for $H$. 
Without loss of generality we assume that $\|S_t\|=1$ for all $t>0$. Notice $P_\open S_a = S_a$.
Let $L=\left[\A^H(\open)S_a\right]\subseteq P_\open H_a$ and $P_L$ be the projection from $P_\open H_a$ onto $L$, which belongs to 
$\A^H(\open)'$.  
We introduce a state $\omega$ of $\A^H(\open)$ by $\omega(x)=\ip{xS_a}{S_a}$.  We have $\omega(x)=\ip{x S_{t_i-s_i}}{S_{t_i-s_i}}$ for any $x \in P_{\open} \A^H(I_i)$.
Now for $x_i\in \A^H(I_{n_i})$ $i=1,2 \cdots N$, we have 
\begin{align*}\omega(P_\open x_1 x_2\cdots x_N) & = \ip{P_\open P_{\open_N}x_1x_2\cdots x_N S_a} {S_a}\\
& =\ip{U_{\open_N} \left(x_1 \ot \cdots x_N \ot 1_{H_{\open_N^c}}\right) U^*_{\open_N} S_a}{S_a}\\
&=\ip{x_1 S_{t_{n_1}-s_{n_1}}}{S_{t_{n_1}-s_{n_1}}}\cdots \ip{x_N S_{t_{n_N}-s_{n_N}}}{S_{t_{n_N}-s_{n_N}}}\\
& = \omega(x_1)\omega(x_2)\cdots \omega(x_N),\end{align*} where $H_{\open_N^c}= \ot_{k=0}^N H_{s_{n_{k+1}}-t_{n_k}}$ with $t_{n_0}=0$ and $s_{n_{N+1}}=a$. This shows that $\omega$ is a product pure state of $\bigotimes_{i=1}^NP_\open\A(I_{n_i})\subset \A^H(\open)$ for 
all $N$. Therefore $\A^H(\open)P_L$ is a type I$_\infty$ factor. 

The other statement, when $t_n<s_{n+1}$, follows from \ref{ArakiWoods}.


(2) For an interval $I$, denote $H_k(I)=\left[ f_1\wedge f_2\wedge  \cdots \wedge f_k: f_i\in L^2(I, \kil^\Comp)  \right]$ the $k-$particle space of the antisymmetric Fock space of $L^2(I)$, and $\kil$ is the multiplicity space of the Clifford flow. Define $$H_\open=\left[\Omega, \xi_{n_1}\wedge\xi_{n_2} \wedge \cdots\wedge \xi_{n_N}: \xi_{n_i}\in H_{2k_i}(I_{n_i}),~ k_i, n_i, N \in \Nat\right].$$  It is not difficult to verify that $\A^{H^e}(\open)$ is nothing but $B(H_\open)$.  
\end{proof}

We denote by $\A^\gamma(\open)$ the local algebra associated with the super product system $H^\gamma$ of an \en-semigroup $\gamma$.

\begin{prop}
Let $\gamma$ and $\beta$ be two \en-semigroups and $\open \subseteq [0,a]$.  Then $$ \A^{\gamma\ot \beta}(\open) =\A^\gamma(\open)\ot \A^\beta(\open).$$
\end{prop}

\begin{proof}
Thanks to Proposition \ref{tensorSPS}, the above proposition holds true for intervals. For elementary sets, it follows from the distributive property of the tensors, and hence for any open set. 
\end{proof}

Let $\theta$ be a generalised CCR flow and $\alpha$ be either a Clifford flow or an even Clifford flow of fixed any fixed rank. From the above proposition it follows immediately, thanks to Lemma \ref{alg-typeI}, 2, that $\A^{\theta\ot \alpha}(\open)$ is a type I factor if and only if $\A^{\theta}(\open)$ is a type I factor, for any open $\open \in [0,a]$.
It is shown in \cite{genccr} that there exists a one parameter continuous family of off-white noises, whose spectral density functions converge to 1 at infinity, such that the associated family of generalised CCR flows $\{\theta_{\lambda}: \lambda \in (0,\frac{1}{2}]\}$  contains mutually non-cocycle-conjugate \en-semigroups. This is accomplished by producing an open set $\open$, for any given  $\lambda_1, \lambda_2  \in (0,\frac{1}{2}]$ such that $\A^{\theta^{\lambda_1}}(\open)$ is a type III factor, but  $\A^{\theta^{\lambda_2}}(\open)$ is a type I factor. From the proceeding discussions we have the following theorem. 

\begin{thm}
There exist uncountably many mutually non-cocycle-conjugate \en-semigroups on the hyperfinite type II$_\infty$ factor of the form $\{\theta^\lambda\ot \alpha:\lambda \in (0,\frac{1}{2}]  \}$, where each $\theta^\lambda$ is a generalised CCR flow arising from off-white noise with spectral density converging to $1$ at infinity, and $\alpha$ is a fixed \en-semigroup, which is   either a Clifford flow or an even Clifford flow of any index.\end{thm}

When the spectral density converges to $\infty$ at $\infty$, the local algebras $\A(\open)$ are not useful in distinguishing the associated generalised CCR flows. Tsirelson used $\liminf$ and $\limsup$ of subspaces of the sum system, associated with elementary sets, to distinguish those generalised CCR flows. Tsirelson's invariants can be equivalently described by $\limsup$ of local von Neumann algebras associated with elementary sets, as shown in \cite{pdct}. We adopt an analogous approach in the context of type II$_\infty$ factors, for tensor products of such \en-semigroups with Clifford flows or even Clifford flows. 

\begin{defn} For a sequence of von Neumann algebras $\A_n \subseteq B(H)$ 
define $$\limsup{\A_n} =\{T\in B(H): \exists ~T_{n_k}\in \A_{n_k}~ \mbox{such that} ~ w-\lim_{k\mapsto \infty}{T_{n_k}}=T \}'',$$ where the
limit of the subsequence $\{T_{n_k}\}$ is taken in the weak operator topology. (We realized this should be termed as $\limsup$ rather than $\liminf$ as  initially defined in \cite{pdct}.) 

Also 
define $$\liminf{\A_n} = \{T\in B(H): \exists ~T_{n}\in \A_{n}~\mbox{s.t.}~s-\lim_{n\rightarrow \infty}{T_{n}}=T,~s-\lim_{n\rightarrow \infty}{T_{n}^*}=T^*\}'',$$ where the
limits of the sequences $\{T_{n}\}$ and  $\{T_{n}^*\}$ are taken in the strong operator topology. 
\end{defn}

Since the local algebras $\A^H(\open)$ for super product systems are not proper von Neumann subalgebras of $B(H_a)$, we need to modify the definition slightly.
 For an elementary open set $\open\subseteq [0,1]$, define $$\tilde{\A}^H(\open) = \A^H(\open)''\cap B(H_1)=\A^H(\open)\oplus \Comp\left(1-P_{\open}\right).$$ For product systems $\tA^H(\open)=\A^H(\open)$.
Given any sequence of elementary open sets $\open_n\subseteq [0,1]$, $\limsup{\tA^H(\open_n)}\subseteq B(H_1)$ is an invariant for the super product system $H = (H_t, U_{s,t})$.

\begin{lem}\label{liminfsup}
For a sequence of von Neumann algebras $\A_n \subseteq B(H)$ $$\limsup{\A_n} \subseteq \left(\liminf{\A_n'}\right)'.$$
\end{lem}

\begin{proof} 
Suppose  $T \in \limsup{\A_n}$ and $S \in \liminf{A_n'}$, so that there exists subsequence $T_{n_k}\in \A_{n_k}$ such that $T_{n_k} \rightarrow T$ weakly, and there exists $S_n \in A_n'$ such that $(S_n,S_n^*) \mapsto (S,S^*)$ strongly. 
Then for any $ \xi, \eta \in H,$ we have \begin{align*} |\ip{ TS \xi} {\eta}- \ip{ T_{n_k}S_{n_k} \xi} {\eta}| & \leq |\ip{ S \xi}{T^*\eta}-\ip{ S \xi}{T_{n_k}^*\eta}|+\|S\xi-S_{n_k}\xi\|\|T_{n_k}^* \eta\|;\\
|\ip{ ST \xi} {\eta}- \ip{ S_{n_k} T_{n_k} \xi} {\eta}| & \leq |\ip{ T \xi}{S^*\eta}-\ip{ T_{n_k} \xi}{S^*\eta}|+\|S^*\eta-S_{n_k}^*\eta\|\|T_{n_k} \xi\|.
\end{align*}
Since $\{\|T_{n_k}^*\eta\|\}$ and $\{\|T_{n_k} \xi\|\}$ are bounded we have 
$$\ip{T S\xi} {\eta}=\lim_{k}\ip{T_{n_k} S_{n_k}\xi} {\eta}= \lim_{k}\ip{S_{n_k}T_{n_k}\xi} {\eta}= \ip{ST\xi} {\eta}~\forall \xi, \eta\in H.$$ 
\end{proof}

For an open set $\open\subseteq [0,1]$ we denote $\open^c$ the interior of the complement in $[0,1]$. Since we are dealing with $L^2$-spaces with respect to Lebesgue measure, end points of the intervals does not matter.  As before $H^e$ denotes the super product system associated with Clifford flow of any rank.

\begin{prop}\label{limsupHe}
Let $\{\open_n:n \in \Nat\}$ be a sequence of elementary sets contained in $[0,1]$ such that $|\open_n|\rightarrow 0$. Then $$\liminf{\tA^{H^e}(\open_n)'}=B(H_1)~~ \mbox{and} ~~ \limsup{\tA^{H^e}(\open_n)}=\Comp.$$
\end{prop}

\begin{proof}
Set $\Gamma^e_a(L^2(\open, \kil^\Comp))  =[\xi_1 \wedge\xi_2 \wedge \cdots \wedge \xi_{2m};~ \xi_1, \xi_2 \cdots \xi_{2m} \in L^2(\open, \kil^\Comp), ~m \in \Nat_0],$ when $m=0$ the wedge product is just the vacuum vector $\Omega$. The map \begin{align*}V_\open((\xi_1 \wedge\xi_2 \wedge \cdots \wedge \xi_{2m}) \otimes (\eta_1 \wedge\eta_2  \wedge \cdots \wedge \eta_{2m'}) )&\\ =\xi_1 \wedge\xi_2  \wedge \cdots \wedge \xi_{2m}\wedge  \eta_1 \wedge  \eta_2\wedge  \cdots \wedge  \eta_{2m'},& \end{align*} where $\xi_1, \xi_2\cdots \xi_{2m} \in L^2(\open, \kil^\Comp)$, $\eta_1, \eta_2 \cdots \eta_{2m'} \in L^2(\open^c, \kil^\Comp)$, extends to an isometry between $\Gamma^e_a(L^2(\open, \kil^\Comp))\ot \Gamma^e_a(L^2(\open^c, \kil^\Comp))\mapsto H^e_1$. Define $$\B(\open)=V_\open \left(B(\Gamma^e_a(L^2(\open, \kil^\Comp)))\ot 1_{\Gamma^e_a(L^2(\open^c, \kil^\Comp))}\right)V_\open^*; ~~ \tilde{\B}(\open)={\B(\open)}''.$$

Since $|\open_n|\rightarrow 0$, for any $f\in L^2((0,1), \kil^\Comp)$, we have$f1_{\open_{n}^c}\rightarrow f$.  Using this it is easy to verify that $\liminf{\tilde{\B}(\open_n^c)}=B(H_1)$. 

Notice that for any elementary set $\open=\cup_{i=1}^N(s_i, t_i)$ $$U_\open\left(\bigotimes _{i=1}^n H^e_{t_i-s_i}\otimes \Phi\right) \subseteq V_{\open}\left( \Gamma^e_a(L^2(\open, \kil^\Comp))\ot \Phi\right),$$ where $\Phi$ denotes the tensor products of vacuum vectors in the remaining tensors and $U_\open$ is the canonical isometry as in \ref{Uopen}. This consequently imply that $\tA^{H^e}(\open)\subseteq \tilde{\B}(\open)$ for any elementary open set $\open \subseteq [0, 1]$.
 Hence  we have $$\tilde{\B}(\open_n^c) \subseteq \tilde{\B}(\open_n)' \subseteq \tA^{H^e}(\open_{n})'~~\forall n \in \Nat.$$ So we have $\liminf{\tA^{H^e}(\open_{n})'}=B(H_1)$.  Now it follows from Lemma \ref{liminfsup} that  $\limsup{\tA^{H^e}(\open_n)}=\Comp$. 
\end{proof}


The Arveson product system of Hilbert spaces associated with generalised CCR flows are described by sum systems. For the definition of sum systems and for the construction of the product systems from sum systems (and also for the definitions/facts/notations regarding $\liminf, \limsup$ of Hilbert subspaces), we ask the reader to refer to  \cite{pdct} and \cite{genccr}. For product systems arising from sum systems also,  end points of an intervals does not matter, while dealing with local algebras (see corollary 25, \cite{pdct}). 

\begin{prop}\label{limsupsum}
Let $H =(H_t, U_{s,t})$ be the product system  constructed from a sum system $(G_{s,t}, S_t)_{s,t \in (0,\infty)}$. For a sequence of elementary sets $\open_n \subseteq [0,1]$, $$\liminf{\A^H(\open_n)}=\{W_0(x+iy): x \in \liminf{G_{\open_n}},~y \in \liminf{G_{\open_n^c}^\perp}\}'';$$ 
$$\limsup{\A^H(\open_n)}=\{W_0(x+iy): x \in \limsup{G_{\open_n}},~y \in \limsup{G_{\open_n^c}^\perp}\}''.$$

Further 
$\limsup{\A^H(\open_n)} = \left(\liminf{\left({\A^H(\open_n)}'\right)}\right)'=\left(\liminf{\A^H(\open_n^c)}\right)'.$
\end{prop}

\begin{proof}  For an elementary set $\open\subseteq [0,1]$, $$\A^H(\open) = \{W_0(x+iy): x \in G_{\open_n},~y \in G_{\open_n^c}^\perp\}''; ~~{\A^\H(\open)}'=\A^H(\open^c),$$ (see section 3, \cite{pdct}). 
The strong continuity of $x\mapsto W_0(x)$ (see \cite{KRP}) implies $$\{W_0(x+iy): x \in \liminf{G_{\open_n}},~y \in \liminf{G_{\open_n^c}^\perp}\}''\subseteq \liminf{\A^H(\open_n)}.$$ On the other hand $\{W_0(x+iy): x \in \liminf{G_{\open_n}},~y \in \liminf{G_{\open_n^c}^\perp}\}'$ 
\begin{align*} & =\{W_0(x+iy): x \in \left(\liminf{G_{\open_n^c}^\perp}\right)^\perp,~y \in \left(\liminf{G_{\open_n}}\right)^\perp\}''\\  & = \{W_0(x+iy): ~x \in \limsup{G_{\open_n^c}} ,~y \in \limsup{G_{\open_n}^\perp}\}''~~\mbox{(by lemma 3.1, \cite{pdct})}\\ 
& \subseteq \limsup{\{W_0(x+iy): x \in G_{\open_n^c} ,~ y \in G_{\open_n}^\perp\}}''~~\mbox{(by lemma 3.2 (i), \cite{pdct})}\\ 
& \subseteq \left(\liminf{\{W_0(x+iy): x \in G_{\open_n^c},~ y \in G_{\open_n}^\perp \}'}\right)'~~ \mbox{(by Lemma \ref{liminfsup})}\\
& = \left(\liminf{\A^H(\open_n)}\right)'.\end{align*} 
Hence $\liminf{\A^H(\open_n)}=\{W_0(x+iy): x \in \liminf{G_{\open_n}},~y \in \liminf{G_{\open_n^c}^\perp}\}''.$  For the proof of the corresponding statement of $\limsup{\A^H(\open_n)}$: one inclusion follows  from lemma 3.2 (i), \cite{pdct}; the other inclusion can be proven  by flipping $\liminf{\A^H(\open_n)}$ with $\limsup{\A^H(\open_n)}$ in the above arguments. The remaining statements follow from above and lemma 3.1, \cite{pdct}.
\end{proof}
 
Let  $(\{T_t^1\}, \{T_t^2\})$ be a perturbation pair and $\theta$ be the associated generalized CCR flow on $B(\Gamma_s(G^\Comp))$.  Let $j:G^\Comp \mapsto G^\Comp$ be the anti-unitary $x+iy\mapsto y+ix$ for $x,y \in G$, and $\Gamma(j): \Gamma_s(G^\Comp)\mapsto \Gamma_s(G^\Comp)$ be the second quantization of $j$ defined by $\Gamma(j)(\varepsilon(\xi))=\varepsilon(j\xi)$ and extended antilinearly to $\Gamma_s(G^\Comp)$. Then $$\Gamma(j) W(x+iy)\Gamma(j)=W(y+ix) ~~ \forall x,y \in G.$$  By the discussion in Example \ref{B(H) example},  the dual \en-semigroup of $\theta$ on $B(\Gamma_s(G^\Comp))$ is conjugate to the \en-semigroup $\ol{\theta}$, given by $$\ol{\theta}_t (W(x+iy)) = \Gamma(j)\theta_t(\Gamma(j)W(x+iy)\Gamma(j))\Gamma(j) = W(T_t^2x +i T^1_ty)~~ \forall x, y \in G.$$ So $\ol{\theta}$ is the  generalized CCR flow given by the perturbation pair $(\{T_t^2\}, \{T_t^1\})$, and in particular the associated Arveson product system $(\ol{H}_t, \ol{U}_{s,t})$   is also given by a sum system, say $(\ol{G}_{s,t}, \ol{S_t})$.  

\begin{cor}\label{limsupgccr}
Let $\theta$ be a generalised CCR flow. Then 
$$\limsup{\tA^\theta(\open_n)} = \left(\liminf{\left({\tA^\theta(\open_n)}'\right)}\right)'.$$
\end{cor}

\begin{proof}
Let $(H_t, U_{s,t})$ be the Arveson's product system of $\theta$. By \ref{H ot barH}, the super product system of $\theta$ is given by $(H_t \ot  \ol{H_t}, U_{s,t} \ot \ol{U_{s,t}})$, which arises from the sum system $\left(G_{s,t} \oplus \ol{G}_{s,t}, S_t \oplus \ol{S_t}\right)$.
Also for any two sequences of Hilbert subspaces $\{G_n\}$ and $\{F_n\}$, it is easy to see that $\liminf{(G_n\oplus F_n)}=\liminf{G_n} \oplus \liminf{F_n}$ and $\limsup{(G_n\oplus F_n)}=\limsup{G_n}\oplus \limsup{F_n}.$
Now the corollary follows from the above Proposition \ref{limsupsum}. 
\end{proof}

\begin{prop}\label{limsuptensors} Let a sequence of elementary sets $\{\open_n\subseteq [0,1]:n \in \Nat\}$ be such that $|\open_n|\rightarrow 0$. Let $\theta$ be any generalised CCR flow and $\alpha$ be either a Clifford flow or an even Clifford flow of any index.
Then $\limsup{\left(\A^\theta(\open_n)\ot \A^\alpha(\open_n)\right)}$ is $\Comp 1$ if and only if $\limsup{\A^\theta(\open_n)}=\Comp 1$.
\end{prop}

\begin{proof} Let $\A_n,\B_n$ be any two families of von Nuemann algebras. It immediately follows if $\limsup{\left(\A_n\ot \B_n\right)}=\Comp$, then both $\limsup{\A_n}=\Comp 1= \limsup{\B_n}$, since $ \limsup{\A_n} \ot \limsup{\B_n}\subseteq \limsup{\left(\A_n\ot \B_n\right)}$.

Also  $ \liminf{\A_n'} \ot \liminf{\B_n'}  \subseteq \liminf{\left(\A_n'\ot \B_n'\right)}.$ Using this and 
Lemma \ref{liminfsup}, we have 
\begin{align*} \limsup{\left(\A_n\ot\B_n\right)}  \subseteq \left(\liminf{\left(\A_n'\ot \B_n'\right)}\right)' & \subseteq \left(\liminf{\A_n'} \ot \liminf{\B_n'}\right)'.\end{align*} 

If $\limsup{\A^\theta(\open_n)}=\Comp 1$ and $|\open_n|\rightarrow 0$ then, thanks to corollary \ref{limsupgccr} and Proposition \ref{limsupHe}, both $\liminf{\left(\A^\theta(\open_n)'\right)} =B(H_1^\theta)$ and  $\liminf{\left(\A^\alpha(\open_n)'\right)}=B(H^e_1)$. Hence $$\limsup{\left(\A^\theta(\open_n)\ot \A^\alpha(\open_n)\right)}=\Comp.$$
\end{proof}


For $r>0$, let $\sigma_r$ be a smooth positive even function with $\sigma_r(\lambda)=\log^r |\lambda|$ for large $|\lambda|$. Then $\sigma_r$ is a spectral density function of an off-white noise, and gives rise to a family of generalised CCR flows $\{\theta^r: r>0\}$. In \cite{T1}, a sequence of elementary sets (with Lebesgue measure converging to $0$) is produced for any given $r_1\neq r_2$, so that $\limsup{\A^{\theta^{r_1}}}(\open_n)=\Comp$ but $\limsup{\A^{\theta^{r_2}}}(\open_n)$ is non-trivial. (Tsirelson produced invariants through sum systems, but this is equivalent to the above statement, as explained in Section 3, \cite{pdct}.) Thanks to Proposition \ref{limsuptensors} we have the following theorem.

\begin{thm}
There exits uncountably many mutually non-cocycle-conjugate \en-semigroups on the hyperfinite type II$_\infty$ factor of the form $\{\theta_r\ot \alpha:r>0\}$, where $\theta_r$ is a generalised CCR flow arising from off-white noise with spectral density $\sigma_r$ converging to $\infty$ at infinity, and $\alpha$ is a fixed \en-semigroup which is  either a Clifford flow or an even Clifford flow of any index.
\end{thm}

\subsection{Tensoring with Toeplitz CAR flows} To study Toeplitz CAR flows discussed in \cite{toepcar}, we need to further specialize the idea of local algebras to the notion of type I factorizations as defined  by Araki and Woods \cite{AW}. Here we define these invariants with respect to super product systems and use them to study \en-semigroups on hyperfinite II$_\infty$ factor, given by tensor products of Toeplitz CAR flows with $\alpha$.  Throughout this subsection, every index set (indexing a type I factorization) is assumed to be countable, and every Hilbert space is assumed to be separable.

\begin{defn} 
Let $H$ be a Hilbert space. 
We say that a family of type I subfactors $\{\m_{\lambda}\}_{\lambda\in \Lambda}$ of $B(H)$ is a \textit{type I factorization} of $B(H)$ if 
\begin{itemize}
\item [(i)] $\m_\lambda\subset \m_\mu'$ for any $\lambda,\mu\in \Lambda$ with $\lambda\neq \mu$, 
\item [(ii)] $B(H)=\bigvee_{\lambda\in \Lambda}\m_\lambda$.
\end{itemize}
We say that a type I factorization $\{\m_{\lambda}\}_{\lambda\in \Lambda}$ is a \textit{complete atomic Boolean algebra of type I factors} 
(abbreviated as \textit{CABATIF}) if for any subset $\Gamma\subset \Lambda$, the von Neumann algebra $\bigvee_{\lambda\in \Gamma}\m_{\lambda}$ is a type I factor. 
\end{defn}

Two type I factorizations $\{\m_{\lambda}\}_{\lambda\in \Lambda}$ of $B(H)$ and $\{\m'_{\mu}\}_{\mu\in \Lambda'}$ of $B(H')$ 
are said to be unitarily equivalent if there exist a unitary $U$ from $H$ onto $H'$ and a bijection $\sigma:\Lambda\rightarrow \Lambda'$ 
such that $U\m_\lambda U^*=\m'_{\sigma(\lambda)}$. For super product systems, we associate type I factorizations of $B(K)$ for a subspace of $K\subseteq H_a$ as follows. 

Let $A=\{a_n\}_{n=0}^\infty$ be a strictly increasing sequence of non-negative numbers starting from $0$ and converging to $a<\infty$. Define $P^A_N= U_N U_N^*$ where $$U_N:\bigotimes_{n=0}^{N-1} H_{a_{n-1}-a_{n}} \ot H_{a-a_N} \mapsto H_a$$ is the canonical isometry uniquely determined by the associativity axiom of the super product system. Clearly $\{P^A_N:N\in \Nat\}$ is a decreasing family of projections in $N$. Define $P^A= \bigwedge_{n=1}^\infty P^A_N.$  (We write $P^{A,\theta}$ to remember the associated \en-semigroup.)

\begin{lem} Let $H=(H_t, U_{s,t})$ be a super product system which can be embedded into a product system, and let $A=\{a_n\}_{n=0}^\infty$ be 
a strictly increasing sequence of non-negative numbers starting 
from $0$ and converging to $a<\infty$. 
Then $\{P^A\A^H((a_n,a_{n+1}))\}_{n=0}^\infty$ is a type I factorization of $B(P^AH_a)$.  
\end{lem}

\begin{proof}
Let $E=(E_t,V_{s,t})$ be a product system where the super product system $H$ can be embedded. Then $\{\A^E((a_n,a_{n+1}))\}_{n=0}^\infty$ is a type I factorization of $B(E(a))$ because  $$B(E_a)=\bigvee_{0<t<a}\A^E((0,t))$$ 
holds (see \cite[Proposition 4.2.1]{Arveson book}). 
Let $Q^A$ be the orthogonal projection from $E_a$ onto $P_AH_a$. Then $Q^A\A^E((a_n,a_{n+1}))Q^A = P^A\A^H((a_n,a_{n+1}))$. 
\end{proof}

The following proposition is immediate,  since $$\bigvee_{\lambda\in \Gamma}\left(\m^1_{\lambda}\ot \m^2_{\lambda}\right) = \left(\bigvee_{\lambda\in \Gamma}\m^1_{\lambda}\right)\otimes \left(\bigvee_{\lambda\in \Gamma}\m_\lambda^2\right).$$

\begin{prop}
For two type I factorizations $\{\m^1_{\lambda}\}_{\lambda\in \Lambda}$ and $\{\m^2_{\lambda}\}_{\lambda\in \Lambda}$, $\{\m^1_{\lambda}\ot \m^2_\lambda\}_{\lambda\in \Lambda}$is a \textit{CABATIF} if and only if both $\{\m^1_{\lambda}\}_{\lambda\in \Lambda}$ and $\{\m^2_{\lambda}\}_{\lambda\in \Lambda}$ are \textit{CABATIF} 
\end{prop}

When $\{\m_\lambda\}_{\lambda\in \Lambda}$ is a type I factorization of $B(H)$, we say that a non-zero vector $\xi$ is \textit{factorizable} 
if for any $\lambda$, there exists a minimal projection $p_\lambda$ of $\m_\lambda$ such that $p_\lambda\xi=\xi$.  Araki and Woods characterized a CABATIF as a type I factorization with a decomposable vector.  One can find the following theorem in \cite[Lemma 4.3, Theorem 4.1]{AW}.  

\begin{thm}[Araki--Woods]\label{ArakiWoods} 
A type I factorization is a \textit{CABATIF} if and only if it has a factorizable vector. 
\end{thm} 

As before let $A=\{a_n\}_{n=0}^\infty$ be a strictly increasing sequence of non-negative numbers starting from $0$ and converging to $a<\infty$. When a super product system $H$ has a unit $\{S_t:t\geq 0\}$, then $P^AS_a=S_a$ and further it gives a factorizable vector for the type I factorization 
$\{P^A\A^H((a_n,a_{n+1}))\}_{n=0}^\infty$, which is necessarily a \textit{CABATIF}  thanks to Theorem \ref{ArakiWoods}. So type I factorization associated with the super product system of Clifford flow of any rank is a \textit{CABATIF} for any sequence $A=\{a_n\}$. 

Now if $\theta$ be a Toeplitz CAR flow and $\alpha$ be either a Clifford flow or an even Clifford flow of any fixed rank, then the type I factorization $\{P^{A,\theta\ot \alpha}\A^{\theta\ot\alpha}((a_n,a_{n+1}))\}_{n=0}^\infty$  is a \textit{CABATIF} if and only if $\{P^{A, \theta}\A^\theta((a_n,a_{n+1}))\}_{n=0}^\infty$ is \textit{CABATIF}.  In \cite{toepcar} an uncountable family of mutual non-cocycle-conjugate Toeplitz CAR flows $\{\theta^\nu: \nu \in (0,\frac{1}{4}]\}$ is constructed. This family is distinguished by providing a sequence $A=\{a_n\}_{n=0}^\infty$ for any given  $ \nu_1, \nu_2 \in (0,\frac{1}{4}]$, so that $\{P^A\A^{\theta^{\nu_1}}((a_n,a_{n+1}))\}_{n=0}^\infty$ is a \textit{CABATIF} but $\{P^A\A^{\theta^{\nu_2}}((a_n,a_{n+1}))\}_{n=0}^\infty$ is not a\textit{CABATIF}. From the above discussions we have the following theorem.

\begin{thm}
There exits uncountably many mutually non-cocycle-conjugate \en-semigroups on the hyperfinite type II$_\infty$ factor of the form $\{\theta_\nu\ot \alpha:\nu \in (0,\frac{1}{4}] \}$, where $\theta_\nu$  is a Toeplitz CAR flow, and $\alpha$ is a fixed \en-semigroup which is either a Clifford flow or even Clifford flow of any index.
\end{thm}

\begin{rem} If a generalised CCR flow (or a Toeplitz CAR flow) is fixed, it is still open to show that it leads to non-cocycle-conjugate \en-semigroups, when tensored with Clifford flows of different indices.

Let $\theta=\{\theta_t:t\geq 0\}$ be a CCR flow of any index and $\psi=\{\psi_t:t\geq 0\}$ be either any of the generalised CCR flow or a Toeplitz CAR flow discussed above, which leads to a type III \en-semigroup on type I factor. Then $\theta\ot \alpha$ is not cocycle conjugate to $\psi\ot \alpha$, since $\theta\ot \alpha$ is multi-spatial, but $\psi\ot \alpha$ does not have any multi-units.  
\end{rem}


\section{CCR flows on hyperfinite type III factors}\label{CCR section}
  
In this section we investigate a class of \en-semigroups on hyperfinite type III factors arising from quasifree representations of the CCR algebra. The structure of these representations was worked out in the early papers \cite{araki}, \cite{Araki2}, \cite{dell antonio}, \cite{holevo}, \cite{Araki Yamagami}, \cite{Araki Woods}; in order to make the paper reasonably self-contained we include the relevant details.
 
For a complex Hilbert space $\Kil$, there exists a universal $C^*$-algebra generated by unitaries $\{w_v:~v\in \Kil\}$, subject to
 $$w_uw_v=e^{-i\im\ip{u}{v}}w_{u+v} \qquad (u,v\in \Kil),$$
 known as the algebra of canonical commutation relations, or CCR algebra, denoted by $CCR(\Kil)$ (see e.g. \cite{petz}). 
 
From here onwards, in the last two sections of this paper, $\kil$ will denote a separable complex Hilbert space, with associated conjugation $j$. 
We denote the conjugation on $\Kil=L^2(\Rplus;\kil)$ also by $j$, obtained as $$(jf)(s):=jf(s) ~\forall s\geq0.$$ Let $A\geq1$ be a complex-linear operator on $\Kil$ such that $T=\frac{1}{2}(A-1)$ is injective. The state on $CCR(\Kil)$ determined  by
 $$\varphi_A(w_f)=e^{-\frac{1}{2}\ip{f}{Af}}=e^{-\frac{1}{2}\norm{\sqrt{1+2T}f}^2}$$
 is known as the quasifree state with symbol $A$.
 The corresponding GNS representation, on $\Gamma_s(\Kil) \otimes \Gamma_s(\Kil)$, is given by
 $$\pi_A(w(f))=W_A(f):=W_0(\sqrt{1+T}f)\otimes W_0(j\sqrt{T}f).$$
 It follows from \cite{araki} that  this representation generates a factor $\m_A=\{\pi_A(w(f)): f \in \Kil\}''$, for which the vacuum vector $\Omega=\varepsilon(0) \otimes \varepsilon(0)$ in $\Gamma_s(\Kil)\ot\Gamma_s(\Kil)=\Gamma_s(\Kil\op\Kil)$ is cyclic and separating. Under this representation $\Omega$ induces the state $\varphi_A$. The factor is type I iff $A^2-1$ is trace class and otherwise it is type III (see \cite{holevo}).

\begin{rems}\label{generalisation remarks}
The case where $A-1$ is not injective can be handled as follows. Since $A-1$ is self-adjoint, the kernel is orthogonal to the range and the operator splits into a direct sum $A_0\oplus A_1$, where $A_0=1$ and $A_1-1$ is injective. The von Neumann algebra obtained from the GNS representation splits into a tensor product $\m_{A_0}\ot\m_{A_1}$, where $\m_{A_0}$ is the type I factor from the Fock representation of $CCR(\Ker(A-1))$ and $\m_{A_1}$ is as above (see \cite{araki} for details). 
\end{rems}

To discuss the Tomita-Takesaki operators for the state $\varphi_A$ we will need to define the second quantisation for unbounded operators. For a closed, densely defined operator $X$ on $\Kil$, set
 $$\Dom \Gamma_0(X)=\Lin\{\varepsilon(f):~f\in\Dom X\},\qquad \Gamma_0(X)\varepsilon(f)=\varepsilon(Xf).$$
 Then $\Gamma_0(X)$ is densely defined and clearly closable, and we denote its closure by $\Gamma(X)$. 
 
 \begin{lem}
 The modular conjugation and modular operator for $(\m_A,\Omega)$ are given by
 $$J_\Omega=\Gamma\begin{bmatrix} 0 & -j \\ -j & 0 \end{bmatrix}, \quad \Delta^{1/2}_\Omega=\Gamma\begin{bmatrix}\sqrt{T}\sqrt{1+T}^{-1} & 0 \\ 0 & j\sqrt{1+T}\sqrt{T}^{-1}j \end{bmatrix}.$$
\end{lem}

 \begin{proof}
  First note that both $\sqrt{1+T}\sqrt{T}^{-1}$ and $\sqrt{T}\sqrt{1+T}^{-1}$ are closed and densely defined, so our candidate operators - call them $J$ and $\Delta^{1/2}$ - make sense. Since $J$ is an anti-involution, and $\Delta^{1/2}$ a positive operator, we only need to check that their product is $S_\Omega$, and the result will follow from uniqueness of the polar decomposition. Define the operator $\Delta^{1/2}_0\subseteq\Delta^{1/2}$ as the positive closed operator with core spanned by $\{W_{A}(f)\Omega:~f\in\Kil\}$. On this set we have
 \begin{align*} & J\Delta^{1/2} W_A(f)\Omega  =e^{-\frac{\|\sqrt{1+T}f\|^2}{2}-\frac{ \|\sqrt{T}f\|^2}{2}}\varepsilon\left(\begin{bmatrix}0&-j\\-j&0\end{bmatrix}\begin{pmatrix}\sqrt{T}f \\ j\sqrt{1+T}f\end{pmatrix}\right)\\ & =e^{-\frac{\|\sqrt{1+T}f\|^2}{2}-\frac{ \|\sqrt{T}f\|^2}{2}}\varepsilon\left(\begin{pmatrix}-\sqrt{1+T} f  \\ -j \sqrt{T} f   \end{pmatrix}\right) =W_A(-f)\Omega. \end{align*} 
Since this set form a core for $S_\Omega$, $J\Delta^{1/2}_0=S_\Omega\subseteq J\Delta^{1/2}$. It follows that $\Delta^{1/2}_0=\Delta^{1/2}_\Omega$, $J=J_\Omega$ and hence the operators
  $$W_A'(f):=J_\Omega W_A(-f)J_\Omega=W_0(\sqrt{T}f)\ot W_0(j\sqrt{1+T}f)$$
  generate $\m_A'$. Also the linear span of the vectors $W_A'(f)\Omega$ form a core for $F_\Omega\supseteq (J_\Omega\Delta^{1/2})^*=\Delta^{1/2}J_\Omega$ and another similar calculation as above reveals that $F_\Omega\subseteq\Delta^{1/2}J_\Omega$, so $\Delta^{1/2}=\Delta^{1/2}_0=\Delta^{1/2}_\Omega$ as required.
 \end{proof}

  We introduce the useful notation
  $$\Sigma_A:=\begin{bmatrix}\sqrt{1+T}&0\\0&j\sqrt{T}j\end{bmatrix},\quad \iota(f):=\begin{pmatrix}f \\\ jf\end{pmatrix}$$
  so that $W_A(f)=W_0(\Sigma_A\iota(f))$. Also let
  $$\Sigma_A':=\begin{bmatrix}\sqrt{T}&0\\0&j\sqrt{1+T}j\end{bmatrix},$$
  so that $W_A'(f)=W_0(\Sigma_A'\iota(f))$.
  
  We wish to define analogues of Arveson's CCR flows on these factors. This requires the following simple Lemma.

  \begin{lem}\label{Toeplitz semigroup lemma} 
   Let $X\in B(L^2(\Rplus;\kil))$ be a positive, injective Toeplitz operator, i.e. $T_t^*XT_t=X$ for all $t\geq0$, where $(T_t)_{t\geq0}$ is the semigroup of right shifts on $L^2(\Rplus;\kil)$. Then the operators $\sqrt{X}T_t\sqrt{X}^{-1}$ extend to a family of isometries which is a strongly continuous semigroup.
  \end{lem}

  \begin{proof}
   Since $\sqrt{X}\geq0$ is injective, $\sqrt{X}^{-1}$ is closed and densely defined. For any $f\in\Dom(\sqrt{X}^{-1})$, we have
   $$\norm{\sqrt{X}T_t\sqrt{X}^{-1}f}^2=\ip{\sqrt{X}^{-1}f}{T_t^*XT_t\sqrt{X}^{-1}f}=\ip{\sqrt{X}^{-1}f}{\sqrt{X}f}=\norm{f}^2,$$
   so that $\sqrt{X}T_t\sqrt{X}^{-1}$ admits a unique isometric extension $Y_t$. For any $f\in\Dom(\sqrt{X}^{-1})$ it is clear that $Y_sY_t f=Y_{s+t}f$ and $Y_tf\to f$ as $t\to 0$, so the family $(Y_t)_{t\geq0}$ is a strongly continuous semigroup of isometries. \end{proof}

For any Toeplitz operator $X$,  we  denote the isometric extension of  $\sqrt{X}T_t\sqrt{X}^{-1}$ by $T^X_t$. The following Proposition assures the existence of \en-semigroups which we call as Toeplitz CCR flow given by the Toeplitz operator $A$. 

  \begin{prop}\label{en construction prop}
   Let $A\geq1$ be a Toeplitz operator on $L^2(\Rplus;\kil)$ such that $A-1$ is injective. Then there exists a unique \en-semigroup $\alpha^A=\{\alpha_t^A: t\geq 0\}$ on $\m_A$ defined by $\alpha_t^A(W_A(f))=W_A(T_tf)$, where $(T_t)_{t\geq0}$ is the semigroup of right shifts on $L^2(\Rplus;\kil)$. Further $\alpha^{A_1\oplus A_2}=\alpha^{A_1}\ot \alpha^{A_2}$.
  \end{prop}
  
  \begin{proof} 
  Thanks to Lemma \ref{Toeplitz semigroup lemma}, we have the semigroup of isometries $$Y=(Y_t)_{t\geq0}=\begin{bmatrix} T_t^{{1+T}} & 0 \\ 0 & jT_t^{{T}}j \end{bmatrix}$$ on $\Kiltwo$. 
   It follows from \cite{Arveson book} Proposition 2.1.3 that there exists a unique \en-semigroup $\sigma$ on $B(H)$ satisfying $\sigma_t(W_0(f))=W_0(Y_tf)$ for all $f\in\Kiltwo$. Clearly $\m_A$ is an invariant subalgebra for $\sigma$ and, by density of $\pi_A(CCR(\Kil))$, the restriction $\alpha_t=\sigma_t|_{\m_A}$ is the unique \en-semigroup satisfying the conditions of the Proposition.
  \end{proof}
  
By construction, each of these \en-semigroups has a faithful, normal invariant state $\varphi_A$, hence a canonical unit $S$. By Proposition \ref{invariant state}, a necessary and sufficient condition for the canonical unit to be a multi-unit is that it commute with the modular conjugation $J_\Omega$. The following Proposition characterises the operators $A$ satisfying this condition.

  \begin{prop}\label{Toeplitz and extension property}
  If $A\geq1$ is a Toeplitz operator such that $A-1$ is injective, and $\alpha$ is the \en-semigroup on $\m_A$ induced by $\alpha_t(W_A(f))=W_A(T_tf)$, then the canonical unit commutes with the modular conjugation if and only if $A=I_{L^2(\Rplus)}\otimes R$ for some $R\in B(\kil)$. Moreover, when this is the case there exists a CCR flow $\sigma$ on $B(H)$ extending both $\alpha$ and $\alpha^{\Omega}$.
 \end{prop}

  \begin{proof}
  Let $S$ be the canonical unit defined by $S_tx\Omega:=\alpha_t(x)\Omega$ for all $t\geq0$, $x\in\m_A$. Since $T$ is a Toeplitz operator we have $\|\sqrt{1+T}f\|=\|\sqrt{1+T}T_tf\|$ and $\|j \sqrt{T}f\|=\|j \sqrt{T}T_tf\|$ for all $f \in \Kil$. By evaluating on $W_A(f)\Omega$ we get
  $$S_t\varepsilon\begin{pmatrix}\sqrt{1+T}f\\ j\sqrt{T}f\end{pmatrix}=\varepsilon\begin{pmatrix}\sqrt{1+T} T_tf\\ j\sqrt{T} T_tf\end{pmatrix}.$$
  By differentiating $s \mapsto S_t\varepsilon\begin{pmatrix}\sqrt{1+T}sf\\ j\sqrt{T}sf\end{pmatrix}$ at $s=0$ we obtain
  $S_t \Sigma_A\iota(f)=\Sigma_A\iota(T_tf),$ and then
 by looking at $\Sigma_A\iota(f)-i\Sigma_A\iota(if)$,
  \begin{equation*} S_t\begin{pmatrix}\sqrt{1+T}f\\ 0\end{pmatrix}=\begin{pmatrix}\sqrt{1+T}T_tf\\ 0\end{pmatrix}. \end{equation*}
  Thus, first by verifying on the domain of $\sqrt{1+T}^{-1}$ and then by extending, we get
  \begin{equation}\label{first T equation}S_t\begin{pmatrix}f\\ 0\end{pmatrix}=\begin{pmatrix}T^{{1+T}}_tf\\ 0\end{pmatrix}~\forall  f\in H.\end{equation}  On the other hand, if $S_tJ_\Omega=J_\Omega S_t$ we get
  $$S_t\Sigma'\iota(f)=S_tJ_\Omega\Sigma\iota(-f)=J_\Omega\Sigma\iota(-T_tf)=\Sigma'\iota(T_tf).$$
  Arguing as before we get  $ S_t\begin{pmatrix}\sqrt{T}f\\0\end{pmatrix}=\begin{pmatrix} \sqrt{T}T_tf\\0 \end{pmatrix}$
  and consequently 
  \begin{equation}\label{second T equation} S_t\begin{pmatrix}f\\ 0\end{pmatrix}=\begin{pmatrix}T^{{T}}_tf\\ 0\end{pmatrix}~\forall  f\in H.\end{equation}
  It follows from equations (\ref{first T equation}) and (\ref{second T equation}) that
  $T_t^{{1+T}}f=T_t^{{T}}f~\forall f \in H,$ which implies \begin{equation}\label{sqrtT} T_t^{{1+T}}\sqrt{T}f=\sqrt{T}T_tf~\forall f \in H.\end{equation} Now if $f \in \Dom(\sqrt{1+T}^{-1})$ then, thanks to equation \ref{sqrtT}, $$\sqrt{1+T}T_t\sqrt{T}\sqrt{1+T}^{-1}f=\sqrt{T}T_tf,$$ and in particular $\sqrt{T}T_tf \in \Dom(\sqrt{1+T}^{-1})$. It follows that $$T_t \sqrt{T}\sqrt{1+T}^{-1}\subseteq \sqrt{T}\sqrt{1+T}^{-1}T_t.$$  
  
So every spectral projection of $\sqrt{T}\sqrt{1+T}^{-1}$ commutes with $T_t$ and hence it is of the form $1_{L^2(\Rplus)}\ot E$ for some projection $E$ in $B(\kil)$. This consequently implies that $\sqrt{1+T}^{-1}\sqrt{T}=1_{L^2(\Rplus)}\ot X$, for some densely defined closed operator (in fact self-adjoint) operator $X$ on $\kil$. Then $(1+T)^{-1}T=1_{L^2(\Rplus)}\ot X^2$. For  $f \in \Dom((1+T)^{-1})$ we have $$(1+T)^{-1}f=(1-T(1+T)^{-1})f=1_{L^2(\Rplus)}\ot (1-X^2)f.$$  This implies that $(1-X^2)$ has a bounded inverse and $$T= 1_{L^2(\Rplus)}\ot \left(1-(1-X^2)^{-1}\right).$$
  
  For the second statement, since $\sqrt{1+T}$ and $j\sqrt{T}$ commute with $T_t$ for all $t\geq0$, the \en-semigroup constructed in Proposition \ref{en construction prop} is the restriction of the usual CCR flow $\theta$ on $B(H)$ of index $2\dim(\kil)$, given by the semigroup of isometries $\begin{bmatrix} T_t & 0 \\ 0 & T_t \end{bmatrix}$. To see that this restricts to $\alpha^\Omega$ on $\m_A'$, we observe that
  $$\alpha^\Omega_t(W_A'(f))=J_\Omega \alpha_t(W_A(-f))J_\Omega=W_A'(T_tf),$$
  but this is equal to
  $$W_A'(\Sigma_A'(T_t\op T_t)\iota(f))=W_A'((T_t\op T_t)\Sigma_A'\iota(f))=\theta_t(W_A'(f)),$$
  as required. 
  \end{proof}
 
  For the rest of the paper we restrict to this class of \en-semigroups, where the canonical unit commutes with the modular conjugation. Since the Toeplitz part of  $A$ is trivial, we just call these \en-semigroups as just CCR flows. We denote the CCR flow given by $A= 1\ot R$ by $\alpha^{(R)}$. Notice that, since $L^2(\Rplus)$ is infinite dimensional, $Tr(I\otimes (R^2-1))<\infty$ iff $Tr(R^2-1)=\norm{\sqrt{R^2-1}}_{HS}=0$, i.e. $R^2-1=0$. By our ($A-1$ is injective) assumption $R\neq1$, so $\m_A$ is a type III factor. 
 
  The second half of the proposition shows that the super-product system for $\alpha,\alpha^\Omega$ is isomorphic to the completely spatial product system of index $2\dim\kil$, hence $\Ind_c(\alpha)=2\dim\kil$. If $\dim\kil=n$, we say that the corresponding \en-semigroup is a CCR flow on the type III factor $\m_A$ of rank $n$. The following Corollary is an immediate consequence of Proposition \ref{coupling ind invariant}.
  
\begin{cor}\label{CCR-rank}
CCR flows on hyperfinite type III factors associated with operators of the form $A_i=1\ot R_i,~i=1,2$  are not cocycle conjugate, if $R_1$ and $R_2$ have different ranks.
\end{cor}

  In order to classify these semigroups further we must determine when the algebras $\m_A$ are isomorphic. For this we require the following lemma.
  
  \begin{lem}\label{tensor lemma}
   Let $X,Y$ be closed, densely defined operators of the form $X=\sum_{i=1}^\infty{\lambda_iP_i}$, $Y=\sum_{j=1}^\infty{\mu_jQ_j}$ where the $\{P_i\}_{i=1}^\infty$ and $\{Q_j\}_{j=1}^\infty$ are families of mutually orthogonal projections. Then $\sigma(X\ot Y)=\ol{\sigma(X)\sigma(Y)}$.
  \end{lem}

  \begin{proof}
   Let $\{e_i\}_{i=1}^\infty$ be an eigenbasis for $X$ and $\{f_j\}_{j=1}^\infty$ an eigenbasis for $Y$ and write $v=\sum_{i,j=1}^\infty{v_{ij}e_i\ot f_j}$. If $v$ is nonzero and $(X\otimes Y)v=\lambda v$ then
   $$\sum_{i,j=1}^\infty{\lambda v_{ij} e_i\ot f_j}=\sum_{i,j=1}^\infty{\lambda_i\mu_j v_{ij} e_i\ot f_j}$$
   so that, for each $i,j$, either $v_{ij}=0$ or $\lambda_i\mu_j=\lambda$, hence $\lambda\in\sigma(X)\sigma(Y)$. If $\lambda\notin\ol{\sigma(X)\sigma(Y)}$ then since $\inf_{i,j} |\lambda-\lambda_i\mu_j|=\varepsilon>0$, we have
   $$ \norm{(X\ot Y-\lambda)v}^2=\sum_{i,j=1}^\infty | \lambda_i\mu_j-\lambda |^2|v_{ij}|^2\geq \varepsilon^2\norm{v}^2, $$
   so $X\ot Y-\lambda$ is bounded below. Thus the only possible values in the spectrum of $X\ot Y$ not in $\ol{\sigma(X)\sigma(Y)}$ are those for which $X\ot Y-\lambda$ does not have dense range. Equivalently $\ol{\lambda}$ is an eigenvalue for $X^*\ot Y^*$, so by the preceding argument $\ol{\lambda}=\ol{\lambda}_i\ol{\mu}_j$ for some $i,j$ and $\lambda\in\sigma(X)\sigma(Y)$.
  \end{proof}

  The following theorem may be gleaned from \cite{Araki Woods}, for the reader's convenience we include the details.
 
 \begin{thm}
  Let $A=I\otimes R \geq 1$ be such that $T=(A-1)/2$ is injective. Then there are the following three possibilities.
  \begin{rlist}
   \item $A$ has discrete spectrum and there exists $\lambda\in (0,1)$ such that the eigenvalues of $(1+T)^{-1}T$ all have the form $\lambda_i=\lambda^{d_i}$ for some $d_i\in\Nat$.
   \item $A$ has discrete spectrum, but is not of the form (i).
  \item $A$ has nonempty purely continuous spectrum (see \cite{Kato} X.1.1).
  \end{rlist}
  In case (i) $\m_A$ is the hyperfinite $III_\lambda$ factor, whereas in all other cases $\m_A$ is the hyperfinite $III_1$ factor.
 \end{thm}

 \begin{proof}
  By definition $A$ is one of the three types described above, so it remains to show the factors are as claimed.
  In Section 12 of \cite{Araki Woods}, the following is observed:
  \begin{itemize}
   \item[(1.)] If $A$ has discrete spectrum then $\m_A$ is an infinite tensor product of factors of type I (ITPFI), so hyperfinite.
   \item[(2.)] If $A$ has discrete spectrum and $\lambda$ is a limit point of $\sigma((1+T)^{-1}T)$, then $\lambda\in r_\infty(\m_A)$, the asymptotic ratio set of $\m_A$.
   \item[(3.)] If $A$ has non-empty purely continuous spectrum then $\m_A$ is isomorphic to an ITPFI and $r_\infty(\m_A)=\Rplus$.
  \end{itemize}
  By \cite{Connes} $r_\infty(\m_A)=S(\m_A)$ for ITPFI factors, so the third point is equivalent to $\m_A$ being hyperfinite type III$_1$. If $A$ has discrete spectrum then, as $A=I\otimes R$, all eigenvalues have infinite multiplicity, so all points in the spectrum are limit points. Thus, if $A$ has discrete spectrum and satisfies (ii), then by (2.) $r_\infty(\m_A) \neq \{0\}\cup\{\lambda^n:~n\in\mathbb{Z}\}$ for any $\lambda\in(0,1)$, and clearly $r_\infty(\m_A)\neq\{0,1\}$, so $S(\m_A)=\Rplus$ and $\m_A$ is type III$_1$. If $A$ satisfies (i) then, again by (2.), $r_\infty(\m_A)\supseteq \{0\}\cup\{\lambda^n:~n\in\mathbb{Z}\}$ and we are left to show that the modular spectrum of $\m$ contains nothing further. We simply show $\sigma(\Delta_\Omega)\supseteq \{0\}\cup\{\lambda^n:~n\in\mathbb{Z}\}$. Since $\Delta_\Omega$ is a sum of tensor powers of $(1+T)^{-1}T\oplus T^{-1}(1+T)$, by Lemma \ref{tensor lemma} its spectrum is the closure of $\bigcup_{n\in\mathbb{Z}}^\infty{\sigma((1+T)^{-1}T)^n}$, that is $\{0\}\cup\{\lambda^n:~n\in\mathbb{Z}\}$. 
 \end{proof}
 
 \begin{rem}\label{types remark}
  When $\kil$ is one dimensional, only (i) can occur. When $\kil$ has finite dimension (ii) occurs iff $(1_T)^{-1}T$ has eigenvalues $\lambda_i,\lambda_j$ with $\log\lambda_i /\log\lambda_j\notin\mathbb{Q}$. In infinite dimensions there are further examples of case (ii) coming from sequences of rational powers with strictly increasing denominators, e.g. $(\lambda^{n/(n+1)})_{n\in\Nat}$. Clearly case (iii) can only occur if $\kil$ is infinite dimensional.
 \end{rem}

 In particular, thanks to Corollary \ref{CCR-rank}, $A=\frac{1+\lambda}{1-\lambda}I_{L^2(\Rplus;\kil)}$ gives infinitely many non-cocycle-conjugate \en-semigroups on each hyperfinite III$_\lambda$ factor with $0<\lambda<1$ distinguished by their rank. Distinguishing between two CCR flows of equal rank is more complicated and we take up a detailed analysis in the next section.
 
 \section{Characterising cocycle conjugacy for CCR flows}\label{CCR section2}
 
  In this section we show that there are uncountably many non-cocycle conjugate \en-semgiroups on each hyperfinite III$_\lambda$ factor with $\lambda\in (0,1]$. The proof relies upon the precise form of the gauge group and a detailed analysis of its fate under cocycle perturbations.
  
  \begin{prop}
   Let $A=I\ot R\geq1$ be such that $A-1$ is injective and consider the corresponding CCR flow $\alpha$ on $\m_A$. 
   Then every element of the gauge group $G(\alpha)$ has the form 
   $$U_t=e^{i\lambda t}W_A(1_{(0,t)}\ot \xi) \qquad (t\geq0)$$
   for some $\lambda\in\Real$, $\xi\in\kil$. As a topological group, $G(\alpha)$ is isomorphic to the central extension of $(\kil,+)$ by the $\Real$-valued 2-cocycle $\omega(\xi,\eta)=-\im\ip{\xi}{\eta}$.
  \end{prop}

  \begin{proof} Let $\theta$ be the CCR flow on $B(H)$ mentioned in Proposition \ref{Toeplitz and extension property}, which extends both $\alpha$ and $\alpha^\Omega$. 
  Since $\alpha_t(\m)'\cap \m\subseteq \theta_t(B(H))'$, every gauge cocycle for $\alpha$ is also a gauge cocycle for $\theta$, and 
$G(\alpha)$ is the subgroup of $G(\theta)$ consisting of cocycles living in $\m_A$. From \cite{Arveson book}, Section 3.8, it follows that $G(\theta)$ consists of cocycles of the form
 $$U_t(\lambda, \xi, V)=e^{i\lambda t}W_0(1_{(0,t)}\ot \xi)(\Gamma(I_{L^2[0,t]}\ot V)\ot \Gamma(I_{L^2([t,\infty);\noisetwo)}))~~~ \forall t\geq0,$$
   where $\lambda\in\Real$, $\xi \in\noisetwo$ and $V\in\mathcal{U}(\noisetwo)$. 
   
If $U_t(\lambda, \xi, V) \in \m_A$ then for any $\eta\in\noise$, we have $$W_A'(1_{(0,t)}\ot \eta)U_t(\lambda, \xi, V)=U_t(\lambda, \xi, V)W_A'(1_{(0,t)}\ot \eta).$$
Evaluating on $\Omega$ we get 
$$W_A'(1_{(0,t)}\ot \eta)W_0(1_{(0,t)}\ot \xi)\Omega=W_0(1_{(0,t)}\ot \xi)W_0((I\ot V) \Sigma'\iota(1_{(0,t)}\ot \eta))\Omega.$$  Thanks to the linear independence of exponential vectors, comparing both sides
$$ \begin{pmatrix}\sqrt{R}\eta\\ j\sqrt{1+R}\eta \end{pmatrix} =V\begin{pmatrix}\sqrt{R}\eta \\j\sqrt{1+R}\eta  \end{pmatrix}; ~~~
\left<\begin{pmatrix}\sqrt{R}\eta \\j\sqrt{1+R}\eta  \end{pmatrix}, \xi\right> = \left<\xi, V\begin{pmatrix}\sqrt{R}\eta \\j\sqrt{1+R}\eta \end{pmatrix}\right>,
$$ for all $\eta\in\kil$. 
The first equation implies that the unitary $V$ is identity on the real liner subspace $L = \left\{\begin{pmatrix}\sqrt{R}\eta\\ j\sqrt{1+R}\eta \end{pmatrix}: \eta \in \kil\right\}$. But the complex Hilbert space spanned by $L$ is whole of $\kil\oplus \kil$ (consider $\begin{pmatrix}\sqrt{R}\eta\\ j\sqrt{1+R}\eta \end{pmatrix}\pm i\begin{pmatrix}\sqrt{R}i\eta\\ j\sqrt{1+R}i\eta \end{pmatrix}$ and $range ( R )$ is dense). The other equation implies that the imaginary part of the inner product of $\xi$ with any element in $L$ is $0$, which means $\xi$ is of the form $\begin{pmatrix}\sqrt{1+R}\eta'\\ j\sqrt{R}\eta' \end{pmatrix}$ for some $\eta' \in \kil$. 
  \end{proof}
 
 \begin{defn}
 A continuous real linear operator $Z:H\mapsto H$ is said to be a symplectic automorphism if  $\im\ip{Zf}{Zg}=\im\ip{f}{g}$ for all $f,g \in H$.
 \end{defn}

 \begin{prop}\label{Z}
  If $R_1\neq ZR_2Z^*$ for any symplectic automorphism $Z$, then the CCR flow corresponding to $A_1=I\ot R_1$ is not cocycle conjugate to the CCR flow corresponding to $A_2=I\ot R_2$.
 \end{prop}

  \begin{proof}
 Let $\alpha^1$ and $\alpha^2$ be the CCR flows corresponding to $R_1$ and $R_2$ acting standardly on $\m_1$ and $\m_2$ respectively. Suppose that there exists a unitary $V$ and an $\alpha^1$-cocycle $(W_t)_{t\geq0}$  implementing cocycle conjugacy so that $Ad_VW_t\alpha^1_tAd_{V^*}=\alpha^2_t,$ for all $t\geq 0.$ Recall the algebras $\semiflowalg_{\alpha^i}(t)=\alpha^i_t(\m_i)'\cap\m_i, ~i=1,2$ defined at the end of Section \ref{prelims}. Note that the isomorphism $\phi_t=Ad_{VW_t}:\semiflowalg_{\alpha^1}(t)\to\semiflowalg_{\alpha^2}(t)$ is strongly continuous. For $i=1,2$ consider the topological group
  $$G_t(\alpha^i):=\{ (u_s)_{s\in [0,t]}:~ (u_s)_{s\geq0}\in G(\alpha^i)\}$$
  which is canonically isomorphic to the gauge group. Then the map $(u_s)_{s\in[0,t]}\to (\phi_t(u_s))_{s\in[0,t]}$ induces an isomorphism $G_t(\alpha^1)\to G_t(\alpha^2)$. Indeed, the only non-obvious aspect is to check that $W_tu_sW_t^*=W_su_sW_s^*$, for each $u\in G_t(\alpha^1)$ and $s\in [0,t]$, which follows from the cocycle property and the fact that $u_s\in\alpha^1_s(\m_1)'$. Denote $c_{t]}=c \ot 1_{(0,t)}$ for any $c \in \kil, t \in \Rplus$.
   Since $\phi_t$ is linear and strongly continuous, there exist continuous maps $\theta:\kil\to\Real$ and $Z:\kil\to\kil$ satisfying $\phi_t(W_{A_1}(c_{s]}))=e^{is\theta(c )}W_{A_2}(Z(c )_{s]})$ for all $c\in \kil$. These induce a group homomorphism, so we must have
  $$e^{is(\theta(c)+\theta(d)-\im\ip{Z(c )}{Z(d)})}W_{A_2}(Z(c )_{s]}+Z(d)_{s]})=e^{is(\theta(c+d)-\im\ip{c}{d})}W_{A_2}(Z(c+d)_{s]}),$$
  hence
  $$ \theta(c)+\theta(d)-\im\ip{Z(c )}{Z(d)}=\theta(c+d)-\im\ip{c}{d}$$
  for all $c,d\in\kil$.
  The imaginary part of the inner products are antisymmetric under an exchange of $c$ and $d$, whereas the other terms are clearly symmetric, thus it follows that $Z$ is a symplectic automorphism and $\theta$ a real linear functional. 
  
  By the Riesz representation theorem there exists $x\in\kil$ with $\theta(c)=\re\ip{x}{c}$ for all $c\in\kil$, and we can form a functional $\Psi$ on $L^2([0,t];\kil)$ by setting $\Psi(f):=\re\ip{1_{[0,t]}\ot x}{f}$. Since $W_{A_j}(c_{[r,s]})=W_{A_j}(-c_{r]})W_{A_j}(c_{s]})$ for each $j=1,2$, $c\in\kil$, $0\leq r\leq s\leq t$ we have, for any step function $f\in L^2([0,t];\kil)$, $\phi_t(W_{A_1}(f))=e^{i\Psi(f)}W_{A_2}((I\otimes Z)f)$, by the homomorphism property of $\phi_t$. Thus, if $f\in L^2([0,t];\kil)$ is the limit of a sequence of step functions $(f_n)_{n=1}^\infty$ then
  $$\phi_t(W_{A_1}(f))=\slim_{n\to\infty}\phi_t(W_{A_1}(f_n))=\slim_{n\to\infty}e^{i\Psi(f_n)}W_{A_2}((I\otimes Z)f_n)=e^{i\Psi(f)}W_{A_2}((I\ot Z)f).$$  
  Now using canonical commutation relations we get 
  $$\left(Ad(W_{A_2}(-i(1_{[0,t]}\otimes x)/2))\phi_t\right)(W_{A_1}(f)) =W_{A_2}((I\ot Z)f).$$
  Since $\phi_t$ is normal this implies that the representations of $CCR(L^2([0,t];\kil))$ given by
  $$w_f\mapsto W_{A_1}(f)\quad\text{and}\quad w_f\mapsto W_{A_2}((I\ot Z)f) \qquad (f\in L^2([0,t];\kil))$$
  are quasi-equivalent. 
  In particular the restriction of $\varphi_{A_1}$ to $CCR(L^2([0,t];\kil))$ is quasi-equivalent to the state
  $$CCR(L^2([0,t];\kil))\ni w_f\mapsto\ip{\Omega}{W_{A_2}((I\ot Z)f)\Omega}=e^{-\frac{1}{2}\re\ip{f}{(I\ot Z^{*})A_2(I\ot Z)f}}.$$
  Thus, by \cite{Araki Yamagami} it must be the case that
  $$\sqrt{I_{L^2([0,t])}\ot R_1}-\sqrt{(I_{L^2([0,t])}\ot Z^{*})(I_{L^2([0,t])}\ot R_2)(I_{L^2([0,t])}\ot Z)}$$
  is Hilbert-Schmidt. But $L^2([0,t])$ is infinite dimensional, so we must have
  $\sqrt{R_1}=\sqrt{Z^{*}R_2Z}$, i.e.
  $$R_1=Z^*R_2Z,$$
  as required.
 \end{proof}

  This condition suggests that there should be a large number of distinct CCR flows on the hyperfinite III$_\lambda$ factor, for each rank $n\geq 2$. To show this, we need to analyse the relation $R_1=Z^*R_2Z$ in more detail.
  
  As a real Hilbert space, $\kil$ is isomorphic to a direct sum $\kil_\Real\oplus \kil_{\Real}$ and under this identification multiplication by $i$ becomes multiplication by $\left[\begin{smallmatrix} 0&-1\\1&0\end{smallmatrix}\right]$. Using this we see that a real-linear operator $X=\left[\begin{smallmatrix} X_1&X_2\\X_3&X_4\end{smallmatrix}\right]\in B(\kil_\Real\oplus\kil_\Real)$ is complex linear iff $X_1=X_4$ and $X_2=-X_3$, and it is a positive complex-linear operator iff it is of the form $X=\left[\begin{smallmatrix} X_1&0\\0&X_1\end{smallmatrix}\right]$, for some positive operator $X$ on $\kil_\Real$. In \cite{KRP} (see Proposition 22.1)  it is shown that for a symplectic automorphism $Z$ there exist unitaries $U_1,U_2$ on $\kil$ and a positive operator $Z_1$ on $\kil_\Real$  such that
  $$U_1^*ZU_2^*=\begin{bmatrix} Z_1&0\\0&Z_1^{-1}\end{bmatrix}.$$
  Setting $U_1^*R_2U_1=\left[\begin{smallmatrix} X&0\\0&X\end{smallmatrix}\right]$ and $U_2R_1U_2^*=\left[\begin{smallmatrix} Y&0\\0&Y\end{smallmatrix}\right]$ we obtain
  $$\begin{bmatrix}Y&0\\0&Y\end{bmatrix}=\begin{bmatrix} Z_1&0\\0&Z_1^{-1}\end{bmatrix}\begin{bmatrix}X&0\\0&X\end{bmatrix}\begin{bmatrix} Z_1&0\\0&Z_1^{-1}\end{bmatrix},$$
  i.e. $Y=Z_1XZ_1$ and $Y=Z_1^{-1}XZ_1^{-1}$, which leads to
  \begin{equation}\label{XYZ equation} Z_1^2XZ_1^2=X\text{ and }Z_1^2YZ_1^2=Y. \end{equation}
  To analyse these conditions we use the following proposition.

  \begin{prop}\label{functional calc prop}
   Suppose $B,R\in B(\kil)$ are positive operators with $R\geq1$. If $BRB=R$ then $B=1$.
  \end{prop}

  \begin{proof}
   First note that if $BRB=R$, then $(R^{-1}BR)B=1$ and $B(RBR^{-1})=1$, so $B$ is invertible with $R^{-1}BR=RBR^{-1}=B^{-1}$. We claim that $\sigma(B)=\sigma(B^{-1})=\sigma(B)^{-1}$. To see this, note that if $B^{-1}-\lambda I$ has inverse $Q$, then
   $$(B-\lambda I) RQR^{-1}=R(R^{-1}BR - \lambda I) QR^{-1}=R(B^{-1}-\lambda I)QR^{-1}=1$$
   and, similarly $RQR^{-1}(B-\lambda I)=1$, so $\lambda\notin\sigma(B)$. Conversely, if $\lambda\notin\sigma(B)$, an almost identical argument suffices to show $\lambda\notin\sigma(B^{-1})$.
   
   Now note that $BR=RB^{-1}$ implies $B^nR=RB^{-n}$ for all $n\in\Nat$, so that $e^{z B}R=Re^{zB^{-1}}$ for all $z\in\Comp$ and hence
   $$\int_{\Real}{e^{z\lambda}d\ip{u}{E_B(\lambda)Rv}}=\int_{\Real}{e^{z\lambda}d\ip{u}{RE_B(\lambda^{-1})v}}$$
   for all $z\in\Comp$, $u,v\in\kil$, where $E_B$ is the spectral measure for $B$. Since their Fourier transforms coincide, the corresponding measures are equal, i.e.
   $$\ip{u}{E_B(\mathfrak{X})Rv}=\ip{u}{RE_B(\mathfrak{X}^{-1})v}$$
   for all $u,v\in\kil$, and Borel sets $\mathfrak{X}\subseteq\sigma(B)$, and where $\mathfrak{X}^{-1}=\{s^{-1}:~s\in\mathfrak{X}\}$. Let $\mathfrak{X}$ be a Borel subset of $(0,1)\cap\sigma(B)$, so that $E_B(\mathfrak{X})$ and $E_{B}(\mathfrak{X}^{-1})$ are orthogonal. Then if $E_B(\mathfrak{X})v=v$, we have
   $$\ip{v}{Rv}=\ip{E_B(\mathfrak{X})v}{Rv}=\ip{v}{RE_B(\mathfrak{X}^{-1})v}=0,$$
   but $R\geq1$, so this implies $v=0$. Since $E_B(\mathfrak{X})=E_B(\mathfrak{X}^{-1})=0$ for all Borel subsets of $(0,1)$ we can infer that $B=1$.
   \end{proof}
   
   Now we are able to give a complete classification of CCR flows when $R-1$ is injective.
   
   \begin{thm}
    Let $R_1,R_2\geq1$ be bounded operators with $R_1-1$ and $R_2-1$ injective. The CCR flows $\alpha^{(R_1)}$ and $\alpha^{(R_2)}$ are cocycle conjugate if and only if there exists a unitary $U$ such that $R_1=UR_2U^*$. When this is true, $\alpha^{(R_1)}$ is conjugate to $\alpha^{(R_2)}$.
   \end{thm}

   \begin{proof}
    If $R_1$ and $R_2$ give cocycle conjugate \en-semigroups, thanks to  Proposition \ref{Z}, $R_1=Z^*R_2Z$ for some symplectic automorphism $Z=U_1\left[\begin{smallmatrix} Z_1 &0\\ 0& Z_1^{-1}\end{smallmatrix}\right]U_2$, where $Z_1$ is a positive operator on $\kil_\Real$. As before if we set $U_1^*R_2U_1=\left[\begin{smallmatrix} X&0\\0&X\end{smallmatrix}\right]$ and $U_2R_1U_2^*=\left[\begin{smallmatrix} Y&0\\0&Y\end{smallmatrix}\right]$
 then $Z_1^2XZ_1^2=X$ and $Z_1^2YZ_1^2=Y$ so by Proposition \ref{functional calc prop}, $Z_1^2=1$. Since $Z_1$ is positive it follows $Z_1=1$.
 
    Conversely, suppose that there exists a unitary $U$ such that $R_1=UR_2U^*$ and let $A_j=I\ot R_j$ for $j=1,2$. Then the quasi-free states given by $A_1$ and $(I\ot U^*)A_2(I \ot U)$ are quasi-equivalent, indeed they are same states. This implies that the representations of $CCR(L^2([0,t];\kil))$ given by
  $$w_f\mapsto W_{A_1}(f)\quad\text{and}\quad w_f\mapsto W_{A_2}((I\ot U^*)f) \qquad (f\in L^2([0,t];\kil))$$
  are quasi-equivalent.   Let $\theta: \m_{A_1}\mapsto \m_{A_2}$ be the isomorphism satisfying $$\theta(W_{A_1}(f))=W_{A_2}((I\ot U^*)f),$$ then,  since $(I \ot U)$ commutes with $T_t$, we have
  $$\left(\theta\alpha_t^{R_1}\theta^{-1}\right)(W_{A_2}(f))= W_{A_2}((I\ot U^*)T_t (I\ot U) f)= W_{A_2}(T_tf)=\alpha_t^{R_2}(W_{A_2}(f)).$$
   \end{proof}
   
   \begin{rems}
    (1) If $0<\lambda<1$ then there exists exactly one rank 1 CCR flow on the hyperfinite $III_\lambda$ factor. 
    
    If $n\geq2$ then there exists a countable infinity of non-cocycle conjugate CCR flows on the hyperfinite $III_\lambda$ factor with rank $n$. These are given, for instance, by choosing natural numbers $1=d_1\leq\ldots\leq d_n$ and then
    $$T(1+T)^{-1}=I\otimes \diag(\lambda^{d_1},\ldots,\lambda^{d_n}),$$
    so that the quasifree representation corresponding to 
    $$R=\diag\left(\frac{1+\lambda^{d_1}}{1-\lambda^{d_1}},\ldots,\frac{1+\lambda^{d_n}}{1-\lambda^{d_n}}\right)$$
    generates a hyperfine $III_\lambda$ factor. Each distinct choice of the $d_i$s gives different eigenvalues for $R$ by injectivity of the map $[0,1)\to \Rplus$, $x\mapsto (1+x)/(1-x)$. 
    
    Using a similar argument we see that there exist uncountably many CCR flows of infinite rank on the hyperfinite III$_\lambda$ factor; one for each distinct sequence of integers $1,d_1, d_2,\ldots$ up to permutations. (To see this collection is uncountable, note that every strictly increasing sequence gives a different example).
    
    \medskip

    (2) The hyperfinite $III_1$ factor admits no CCR flows of rank 1. For any rank $n\geq2$, the hyperfinite type III$_1$ factor admits uncountably many non-cocycle conjugate CCR flows. For $n$ finite this is seen by noting that each distinct sequence of numbers $\lambda_1\leq \ldots\leq \lambda_n$ in $(0,1)$ for which at least one pair $(\lambda_i,\lambda_j)$ satisfies
    $$\log(\lambda_i)/\log(\lambda_j)\notin \mathbb{Q}$$
    defines a CCR flow on the hyperfinite III$_1$ factor with 
    $$R=\diag\left(\frac{1+\lambda_1}{1-\lambda_1},\ldots,\frac{1+\lambda_n}{1-\lambda_n}\right).$$
    When $n=\infty$ there exist further examples as indicated by Remark \ref{types remark}.
    
    \medskip
    
    (3) For a positive contraction $S$ on $\kil$, satisfying $Ker(S)=\{0\}=Ker(I-S)$,  consider the quasi-free state on the CAR algebra $\A(L^2((0,\infty), \kil))$, given by $A=I \ot S$. When $S \neq \frac{1}{2}$, the von Neumann algebra $\m_A=\pi_A(\A(L^2((0,\infty), \kil)))$ is a type III factor. The association $$\alpha_t(\pi_A(a(f))) \mapsto \pi_A(a(T_tf))$$ extends to an \en-semigroup on $\m_A$, which is in  standard form. 
It can be proven that this $\alpha$ is equi-modular, that is it satisfies the conditions in Proposition \ref{invariant state}. Hence the  vacuum unit is a multi-unit for $\alpha$.  
The relative commutant $\alpha_t(\m_A)'\cap \m_A$ equals to $\pi_A(\A(L^2((0,t), \kil))_e)''$, the von Nuemann algebra generated by the even products in $\A(L^2((0,t), \kil))$.  This fact about the relative commutants, possibly known to experts, can be found in \cite{Bk}. Since $\alpha_t(\m_A)'\cap \m_A$ and $\alpha_t(\m_A)$ together do not generate $\m_A$, it follows that $\alpha_t$ is not canonically extendable (see Theorem 3.7, \cite{BISS}).   
    
    The CCR  flows on type III factors given by operators of the form $1\ot R$ are canonically extendable, as proved in Proposition \ref{Toeplitz and extension property}. By Proposition \ref{extn}, these CCR flows are not cocycle conjugate to any of the above mentioned CAR flows.
   \end{rems}


\end{document}